\documentclass[
reqno,
10pt,
oneside
]{article}

\usepackage[ngerman, english]{babel}
\usepackage[utf8]{inputenc}
\usepackage[normalem]{ulem}
\usepackage[babel,french=guillemets,german=swiss]{csquotes}
\usepackage[center,font=small]{caption}
\usepackage{cite}
\usepackage{bbm}
\usepackage[raggedright]{titlesec}
\usepackage{xcolor}
\usepackage{enumerate}

\titleformat{\section}{\normalfont\large\bfseries}{\thesection}{1em}{}
\titleformat{\subsection}{\normalfont\bfseries}{\thesubsection}{1em}{}

\usepackage{tocbasic}
\DeclareTOCStyleEntry[
beforeskip=.2em plus 1pt,
]{tocline}{section}

\usepackage[%
left=1.5in,
right=1.5in,
bottom=1in,
top=0.8in,
footskip=0.4in,
]{geometry}

\usepackage{color}
\definecolor{LinkColor}{rgb}{0,0,1}
\definecolor{LinkColor2}{rgb}{0,0.5,0}
\definecolor{lbcolor}{rgb}{0.85,0.85,0.85}
\definecolor{FrameColor}{rgb}{0.85,0.85,0.85}
\definecolor{rosso}{rgb}{0.8,0,0}
\definecolor{lightgray}{rgb}{0.5,0.5,0.5}
\definecolor{violet}{rgb}{0.65,0,0.65}
\definecolor{darkgreen}{rgb}{0,0.5,0}

\usepackage{enumitem}

\usepackage{graphicx}
\usepackage{placeins}
\usepackage{overpic}

\usepackage{amsmath}
\usepackage{amssymb}
\usepackage{dsfont}
\usepackage{empheq}
\usepackage{amsthm}

\usepackage[%
pdftitle={Titel},%
pdfauthor={Autor},%
pdfcreator={LaTeX, LaTeX with hyperref and KOMA-Script},
pdfsubject={Betreff}, 
pdfkeywords={Keywords}
]{hyperref} 

\hypersetup{%
	colorlinks	=true,
	linkcolor	=LinkColor,%
	anchorcolor	=LinkColor,%
	citecolor	=LinkColor2,%
	filecolor	=LinkColor,%
	menucolor	=LinkColor,%
	urlcolor	=LinkColor,%
}
\usepackage{marginnote}

\usepackage{verbatim}

\newtheorem{theorem}{Theorem}[section]
\newtheorem{lemma}[theorem]{Lemma}
\newtheorem{proposition}[theorem]{Proposition}

\theoremstyle{definition}
\newtheorem{remark}[theorem]{Remark}

\makeatletter
\renewenvironment{proof}[1][\proofname]{%
	\par\pushQED{\qed}\normalfont%
	\topsep6\p@\@plus6\p@\relax
	\trivlist\item[\hskip\labelsep\bfseries#1\@addpunct{.}]%
	\ignorespaces
}{%
	\popQED\endtrivlist\@endpefalse
}
\makeatother

\makeatletter
\renewcommand\paragraph{\@startsection{paragraph}{4}{\z@}%
	{1ex \@plus1ex \@minus.2ex}%
	{-1em}%
	{\normalfont\normalsize\bfseries}}
\renewcommand\subparagraph{\@startsection{paragraph}{4}{\z@}%
	{1ex \@plus1ex \@minus.2ex}%
	{-1em}%
	{\normalfont\normalsize\itshape}}
\makeatother

\newcommand{\comm}[1]{#1}
\renewcommand{\comm}[1]{}

\newcommand{\abs}[1]{\left| #1 \right|}
\newcommand{\bigabs}[1]{\big| #1 \big|}

\newcommand{\norm}[1]{\| #1 \|}
\newcommand{\bignorm}[1]{\big\| #1 \big\|}

\newcommand{\R}{\mathbb R}
\newcommand{\N}{\mathbb N}
\newcommand{\n}{\mathbf{n}}

\newcommand{\intO}{\int_\Omega}

\newcommand{\dx}{\;\mathrm{d}x}
\newcommand{\dy}{\;\mathrm{d}y}
\newcommand{\dt}{\;\mathrm dt}
\newcommand{\ds}{\;\mathrm ds}

\newcommand{\ddt}{\frac{\mathrm d}{\mathrm dt}}

\newcommand{\delt}{\partial_{t}}

\newcommand{\deln}{\partial_\n}

\newcommand{\Grad}{\nabla}
\newcommand{\Lap}{\Delta}
\newcommand{\Div}{\operatorname{div}}

\newcommand{\emb}{\hookrightarrow}

\newcommand{\ov}{\overline}
\newcommand{\suchthat}{\;\ifnum\currentgrouptype=16 \middle\fi|\;}

\newcommand{\LL}{\mathbf{L}}
\newcommand{\HH}{\mathbf{H}}
\newcommand{\Om}{\Omega}
\newcommand{\TTn}{\mathbb{T}^n}

\newcommand{\vv}{\mathbf{v}}
\newcommand{\tvv}{\tilde{\mathbf{v}}}
\newcommand{\tp}{\tilde{p}}
\newcommand{\tc}{\tilde{c}}
\newcommand{\tmu}{\tilde{\mu}}

\newcommand{\eps}{\varepsilon}
\newcommand{\vveps}{\vv_\eps}
\newcommand{\ceps}{c_\eps}
\newcommand{\mueps}{\mu_\eps}
\newcommand{\peps}{p_\eps}
\newcommand{\vvepsk}{\vv_\eps^k}
\newcommand{\cepsk}{c_\eps^k}
\newcommand{\muepsk}{\mu_\eps^k}
\newcommand{\pepsk}{p_\eps^k}

\newcommand{\Vsigma}{\HH^1_{\sigma}(\Omega)}

\begin{document}

%
%

\title{\bfseries 
    Nonlocal-to-local convergence rates \\
    for strong solutions to a Navier-Stokes-Cahn-Hilliard system \\with singular potential
       
}

\author{
    Christoph Hurm\footnotemark[1]
    \and Patrik Knopf\footnotemark[1]
    \and Andrea Poiatti\footnotemark[2]
    }

\date{ }

\maketitle

\renewcommand{\thefootnote}{\fnsymbol{footnote}}

\footnotetext[1]{
    Faculty for Mathematics, 
    University of Regensburg, 
    93053 Regensburg, 
    Germany \newline
	\tt(%
        \href{mailto:christoph.hurm@ur.de}{christoph.hurm@ur.de},
        \href{mailto:patrik.knopf@ur.de}{patrik.knopf@ur.de}%
        ).
}

\footnotetext[2]{
	Faculty of Mathematics, University of Vienna, Oskar-Morgenstern-Platz
1, A-1090 Vienna, Austria\newline
	\tt(%
	\href{mailto:andrea.poiatti@univie.ac.at}{andrea.poiatti@univie.ac.at}%
	).
}

\vspace{-2ex}
\begin{center}
	\scriptsize
	\color{white}
	{
		\textit{This is a preprint version of the paper. Please cite as:} \\  
		C.~Hurm, P.~Knopf, A.~Poiatti, 
        \textit{[Journal]} \textbf{xx}:xx 000-000 (2024), \\
		\texttt{https://doi.org/...}
	}
\end{center}

\medskip

%
%

\begin{small}
\noindent\textbf{Abstract.}
The main goal of this paper is to establish the nonlocal-to-local convergence of strong solutions to a Navier--Stokes--Cahn--Hilliard model with singular potential describing immiscible, viscous two-phase flows with matched densities, which is referred to as the Model~H. This means that we show that the strong solutions to the nonlocal Model~H converge to the strong solution to the local Model~H as the weight function in the nonlocal interaction kernel approaches the delta distribution. Compared to previous results in the literature, our main novelty is to further establish corresponding convergence rates. Before investigating the nonlocal-to-local convergence, we first need to ensure the strong 
well-posedness of the nonlocal Model~H. In two dimensions, this result can already be found in the literature, whereas in three dimensions, it will be shown in the present paper. Moreover, in both two and three dimensions, we establish suitable uniform bounds on the strong solutions of the nonlocal Model~H, which are essential to prove the nonlocal-to-local convergence results.
\\[1ex]
\\[1ex]	

\end{small}

\begin{small}
\setcounter{tocdepth}{2}
\hypersetup{linkcolor=black}
\tableofcontents
\end{small}

\setlength\parskip{1ex}
\allowdisplaybreaks
\numberwithin{equation}{section}
\renewcommand{\thefootnote}{\arabic{footnote}}

\newpage

\section{Introduction} 
\label{SECT:INTRO}

The mathematical description of two-phase flows is an important but very challenging topic of modern fluid dynamics with various applications in biology, chemistry and engineering. The motion of a mixture of two immiscible fluids, both having a constant individual density, can essentially be captured by describing the motion of the interface separating the fluids. Therefore, two fundamental mathematical approaches have been developed: \textit{sharp-interface models} and the \textit{diffuse-interface models}.
In sharp-interface models, the interface between the fluids is represented as an evolving hypersurface, which leads to a free boundary problem.
In diffuse-interface models (also referred to as phase-field models), the interface is approximated by a thin tubular neighborhood. The concentrations (or volume fractions) of the two fluids are represented by an order parameter, the so-called phase-field. Except at the diffuse interface, this phase-field will attain values close to $-1$ or $1$ as these values respresent the two fluids, respectively. At the diffuse interface, we expect the phase-field to exhibit a continuous transition between $-1$ and~$1$.
The main advantage of this method is that the time evolution of the phase-field can be described by a PDE system in Eulerian coordinates. This avoids directly tracking the interface as it needs to be done in free boundary problems. In many cases, diffuse-interface models can be related to a corresponding sharp-interface model by the so-called sharp-interface limit, where the interfacial width is sent to zero. For more details on the two approaches, especially in the context of two-phase flows, we refer to \cite{AbelsGarckeReview} and the references therein.

In this work, we investigate a Navier-Stokes-Cahn-Hilliard system known as the \textit{Model~H} (both its local and its nonlocal version), which is a diffuse-interface model describing the time evolution of two immiscible, viscous fluids with \textit{matched densities}. This means that the individual densities of the two fluids can be approximately considered as equal.

Depending on a parameter $\eps\ge 0$, the local and the nonlocal version of the Model~H can be formulated simultaneously as follows.
For $n\in\{2,3\}$, let $\Omega$ either be a bounded domain in $\R^n$, whose boundary $\Gamma:=\partial\Omega$ is of class $C^3$, or let $\Omega$ be the $n$-dimensional torus $\TTn := \big[ \R / \big((2\mathbb Z + 1)\pi\big) \big]^n$.
For any final time $T>0$, we write
$\Omega_T := \Omega\times(0,T)$ and if $\Omega$ is a bounded domain, we further use the notation $\Gamma_T := \Gamma \times (0,T)$.
Then, the following Navier--Stokes--Cahn--Hilliard system is referred to as the \textit{Model~H}:
\begin{subequations}
\label{ModelH}
\begin{align}
    \label{ModelH:NS}
	&\rho\big(\delt \vv + (\vv \cdot \Grad)\vv\big) - \nu \Lap\vv + \Grad p = \mu\Grad c,
    \quad \Div(\vv) = 0
    &&\text{in}\;\Omega_T,
    \\
    \label{ModelH:CH1}
	&\delt c + \vv\cdot\Grad c = m\Delta\mu\qquad&&\text{in}\;\Omega_T, 
    \\
    \label{ModelH:CH2}
	&\mu = \mathcal{L}_\eps c + f^\prime(c)\qquad&&\text{in}\;\Omega_T,
    \\
    \label{ModelH:IC}
    &\vv\vert_{t=0} = \vv_0, \quad c\vert_{t=0} = c_0
    &&\text{in}\; \Omega.
\end{align}
In case $\Omega$ is a bounded domain, we further impose the standard boundary conditions
\begin{align}
    \label{ModelH:BC}
    \vv = 0,\quad \deln \mu =0
    \quad\text{on}\; \Gamma_T,
\end{align}
\end{subequations}
and if $\Omega=\TTn$, we assume periodic boundary conditions.
Here, $\vv:\Omega_T\to\R^n$ denotes the \textit{velocity field} associated with the mixture of two fluids, $p:\Omega_T\to\R$ represents the corresponding \textit{pressure}, $c:\Omega_T\to \R$ is the \textit{phase-field} and $\mu:\Omega_T\to\R$ denotes the \textit{chemical potential}. 
The quantities $\rho$, $\nu$ and $m$ represent the \textit{mass density} of the mixture, the \textit{kinematic viscosity}, and the \textit{mobility}, respectively, which are all assumed to be positive constants.
The function $f'$ is the derivative of a potential $f$, which is usually double-well shaped. A physically relevant choice is the logarithmic potential
\begin{equation}
    \label{DEF:FLOG}
    f_\mathrm{log}(s) := \frac{\theta}{2} \big[(1+s)\,\ln(1+s) +(1-s)\,\ln(1-s)\big]
            - \frac{\theta_0}{2}(1-s^2)
\end{equation}
for all $s\in (-1,1)$,
which is also referred to as the \textit{Flory--Huggins potential}. It is classified as a singular potential as its derivative tends to $\pm\infty$ as its argument approaches $\pm1$. In our mathematical analysis, we will even be able to handle a more general class of singular potentials that will be specified by the assumptions \ref{ASS:S1}--\ref{ASS:S3}.

For any $\eps\ge 0$ and a sufficiently regular function $u:\Omega\to\R$, the operator $\mathcal{L}_\eps$ appearing in \eqref{ModelH:CH2} is defined as
\begin{align}
    \label{DEF:NL:INTRO}
	\mathcal{L}_\eps u(x) := 
    \left\{
    \begin{aligned}
         &\intO J_\eps(|x-y|) \big( u(x) - u(y) \big) \dy
         &&\text{if $\eps>0$},
         \\[1ex]
         &-\Lap u(x)
         &&\text{if $\eps=0$},
    \end{aligned}
    \right.
\end{align}
for all $x\in\Omega$. Here, $J_\eps$ is a suitable nonnegative interaction kernel, whose exact properties will be specified in \ref{ASS:JEPS}. 

In the case $\eps=0$, $\mathcal{L}_0 = -\Lap$ is a local differential operator, where $\Lap$ denotes the Laplace operator subject to the homogeneous Neumann boundary condition $\deln u = 0$ on $\partial\Omega$ if $\Omega$ is a bounded domain, and the Laplace operator with periodic boundary conditions if $\Omega$ is the torus $\TTn$.
Therefore, system \eqref{ModelH} with $\eps=0$ will be called the \textit{local Model~H}. 
It has already been proposed in \cite{HohenbergHalperin}, and a mathematical derivation was provided later in \cite{GurtinPolignoneVinals}. For the analysis of the local Model~H, we refer to \cite{Abels2009, Boyer, GalGrasselli2010,GMT2019} and references therein.
Variants of the local Model~H with dynamic boundary conditions, which allow for a better description of short-range interactions between the fluids on the boundary of the domain, have been proposed an analyzed, for instance, in \cite{GGM2016,GGW2019,Giorgini2023,GGP}.
We further point out that a generalization of the local Model~H that also covers the situation of \textit{unmatched densities} (i.e., both fluids may have different individual densities) has been derived in \cite{AGG} and is known as the \textit{AGG Model}. It has been analyzed, for instance, in \cite{Abels2023,abels2013existence,abels2013incompressible,AbelsWeber2021,Giorgini2021, giorgini2022-3D}. A variant of the same model allowing to treat multi-phase fluids is also studied in \cite{AGP}. 

In the case $\eps>0$, $\mathcal{L}_\eps$ is a nonlocal operator since for any $x\in\Omega$, $\mathcal{L}_\eps u(x)$ depends on all values $u(y)$ with $y\in\Omega$. 
Therefore, system \eqref{ModelH} with $\eps>0$ is referred to as the \textit{nonlocal Model~H}. 
In contrast to the local Model~H, where only short-range interactions between the fluids are taken into account by the differential operator $\mathcal{L}_0$, the nonlocal Model~H also describes long-range interactions between the materials, which are weighted by the interaction kernel $J_\eps$.
To the best of our knowledge, the nonlocal Model~H has first been investigated in \cite{Colli2012}. Afterwards, it has further been analyzed, for instance, in \cite{Frigeri2012,Frigeri2012a,Frigeri2013,Frigeri2015,Frigeri2016,Frigeri2021,AGGP}. 

The system \eqref{ModelH} is associated with the energy functional
\begin{align}
    \label{DEF:EN}
    E_\eps(\vv,c) 
    =  \intO \frac \rho 2 |\vv(x)|^2 \,\mathrm{d}x 
    + \mathcal{E}_\eps(c)
    + \intO f\big(c(x)\big) \,\mathrm{d}x
\end{align}
where, depending on the choice of $\eps$, the contribution $\mathcal{E}_\eps(c)$ is defined as
\begin{align}
    \label{DEF:EN:CH}
    \mathcal{E}_\eps(c)
    := 
    \left\{
    \begin{aligned}
         &\frac14\intO\intO J_{\eps}(x-y)\big| c(x) -  c(y)\big|^2\:\mathrm{d}y\mathrm{d}x
         &&\text{if $\eps>0$},
         \\[1ex]
         &\frac12 \intO\intO \big| \Grad c(x)\big|^2 \,\mathrm{d}x
         &&\text{if $\eps=0$}.
    \end{aligned}
    \right.
\end{align}
The first summand on the right-hand side of \eqref{DEF:EN} represents the kinetic energy, whereas the last two summands in \eqref{DEF:EN} represent the free energy of the mixture, which is either of Ginzburg--Landau type ($\eps=0$) or of Helmholtz type ($\eps>0$). 

For any $\eps\ge 0$, sufficiently regular solutions of the Model H \eqref{ModelH} satisfy the \textit{mass conservation law}
\begin{align}
    \label{mass}
    \intO c(t) \dx = \intO c_0 \dx
    \quad\text{for all $t\in [0,T]$}
\end{align}
as well as the \textit{energy dissipation law}
\begin{align}
    \label{energy}
    \ddt E_\eps\big(\vv(t),c(t)\big)
    = - \nu \intO \abs{\Grad \vv(t)}^2 \dx 
        - m \intO \abs{\Grad \mu}^2 \dx
    \quad\text{for all $t\in [0,T]$.}
\end{align}
In the case $\Omega=\TTn$, we further have
\begin{align}
    \label{avg}
    \intO \vv(t) \dx = \intO \vv_0 \dx
    \quad\text{for all $t\in [0,T]$.}
\end{align}

As shown in \cite{Ponce,Ponce2}, the nonlocal energies $E_\eps$ with $\eps>0$ and the local energy $E_0$ can be related via the nonlocal-to-local convergence
\begin{align}
    \label{NLTL:EN}
    \mathcal{E}_\eps(c) \to \mathcal{E}_0(c)
    \qquad\text{as $\eps\searrow 0$,}
\end{align}
provided that $c\in H^1(\Omega)$. Based on this result, the nonlocal-to-local convergence
\begin{align}
    \label{NLTL:LE}
    \mathcal{L}_\eps(c) \to \mathcal{L}_0(c)
    \qquad\text{as $\eps\searrow 0$}
\end{align}
as well as corresponding nonlocal-to-local convergence results (without rates) for the Cahn--Hilliard equation were established in \cite{DST,DST2,DST3,EJ,Melchionna2019}. The nonlocal-to-local convergence of the Model~H (without rates) has already been established in \cite{Abels2022}. In fact, even the more general case of unmatched densities was covered there.

Recently, in \cite{abels2023strong}, stronger nonlocal-to-local convergence results for the operator $\mathcal{L}_\eps$ (which will be recalled in Proposition~\ref{PROP:NLTL}) were obtained and applied to the Allen--Cahn equation and the Cahn--Hilliard equation. The most substantial improvement of these new results is that concrete rates for the convergence \eqref{NLTL:LE} could be shown. However, compared to previous results, higher regularity of the function $c$ is required.

\paragraph{Outline of this paper.}
In the present contribution, we intend to prove the nonlocal-to-local convergence of the Model~H~\eqref{ModelH} along with corresponding convergence rates. Therefore, in order to apply the convergence results established in \cite{abels2023strong}, we have to consider strong solutions of both the local and the nonlocal Model~H. In this regard, the coupling with the Navier--Stokes equation leads to additional difficulties compared to the results for nonlocal-to-local convergence of the Cahn--Hilliard equation. For example, in three dimensions, we can merely expect local-in-time existence of strong solutions to the Model~H as the global existence of strong solutions of the Navier--Stokes equation is a well-known open problem.

Concerning the existence and uniqueness of weak and strong solutions, the local Model~H is already very well understood. It is also known that strong solutions satisfy the so-called \textit{strict separation property}. This means that the phase-field attains its values only in a strict subinterval of $(-1,1)$ and is thus separated from the pure phases that are associated with $\pm1$ (see~\eqref{SP-Starloc} and~\eqref{SP-Starloc2}). For more details about separation properties of the local Cahn--Hilliard equation, we refer to \cite{MZ,GGM,Gal2023,GKW,GP}. 
All the aforementioned results will be recalled in Proposition~\ref{PROP:WP:LOCAL}. 
We point out that the strict separation property of strong solutions to the local Model~H will be a crucial ingredient in the proof of nonlocal-to-local convergence.

For the nonlocal Model~H, the existence of a weak solution has already been proven in \cite{Frigeri2021}. Moreover, in two dimensions, the strong well-posedness has been established in \cite{GGG2017} (see also \cite{AGGP}). However, apparently, the strong well-posedness in three dimensions has not yet been addressed in the literature. Therefore, in Theorem~\ref{THM:WP:NONLOCAL}, we collect the existence and uniqueness results in two dimensions and we prove the local-in-time strong well-posedness in three dimensions. Moreover, we establish certain bounds on weak and strong solutions, which are independent of the parameter $\eps$. These uniform bounds will be essential in the proof of nonlocal-to-local convergence for strong solutions. We further show that for any $\eps>0$, the strong solution satisfies a strict separation property, which holds as long as the solution exists. However, we are not able to exploit this strict separation result to prove the nonlocal-to-local convergence, since the confinement interval is not uniform in $\eps$. We will thus resort to a different technique. 
The proof of Theorem~\ref{THM:WP:NONLOCAL} is presented in Section~\ref{PROOF:WELLPOSEDNESS}.
For more details about separation properties for the nonlocal Cahn--Hilliard equation, we refer to \cite{GGG2017,P}, in which the first results in 2D and 3D, respectively, are shown (see also, for instance, \cite{Gal2023,GP} and references therein).

Eventually, in Theorem~\ref{THM:NLTL}, we establish the nonlocal-to-local convergence of strong solutions to the Model~H \eqref{ModelH} as $\eps\to 0$ along with associated convergence rates. The proof of this result is presented in Section~\ref{PROOF:NLTL}.


\section{Notation and preliminaries} 
\label{SECT:PRELIM}

In this section, we introduce some notation, assumptions and preliminaries that are supposed to hold throughout the remainder of this paper.

\subsection{Notation}

We start by introducing some notation.

\begin{enumerate}[label=\textnormal{(N\arabic*)},leftmargin=*]

\item \textbf{Notation for general Banach spaces.} 
For any normed space $X$ of scalar-valued functions, we denote its norm by $\|\cdot\|_X$,
its dual space by $X^*$ and the duality pairing between $X^*$ and $X$ by $\langle\cdot,\cdot\rangle_X$.
Besides, if $X$ is a Hilbert space, we write $(\cdot,\cdot)_X$ to denote the corresponding inner product.
Furthermore, for any vector space $X$, corresponding spaces of vector-valued or matrix-valued functions with each component belonging to $X$ are denoted by $\mathbf{X}$.

\item \textbf{Lebesgue and Sobolev spaces.} 
For any $n\in\N$, let now $\Omega$ be either a bounded domain in $\R^n$ of class $C^3$ or the torus 
$\TTn$, which accounts for periodic boundary conditions.
For $1 \leq p \leq \infty$ and $k \in \N$, the standard Lebesgue spaces and Sobolev spaces defined on $\Omega$ are denoted by $L^p(\Omega)$ and $W^{k,p}(\Omega)$, and their standard norms are denoted by $\|\cdot\|_{L^p(\Omega)}$ and $\|\cdot\|_{W^{k,p}(\Omega)}$, respectively.
In the case $p = 2$, we use the notation $H^k(\Omega) = W^{k,2}(\Omega)$. We point out that $H^0(\Omega)$ can be identified with $L^2(\Omega)$.
For simplicity, we just write $(\cdot,\cdot) := (\cdot,\cdot)_{L^2(\Omega)}$, $\|\cdot\|:=\|\cdot\|_{L^2(\Omega)}$ and $\langle\cdot,\cdot\rangle := \langle\cdot,\cdot\rangle_{H^1(\Omega)}$.

Moreover, for any interval $I\subset\R$, any Banach space $X$, $1 \leq p \leq \infty$ and $k \in \N$, we write $L^p(I;X)$, $W^{k,p}(I;X)$ and $H^{k}(I;X) = W^{k,2}(I;X)$ to denote the Lebesgue and Sobolev spaces of functions with values in $X$. The standard norms are denoted by $\|\cdot\|_{L^p(I;X)}$, $\|\cdot\|_{W^{k,p}(I;X)}$ and $\|\cdot\|_{H^k(I;X)}$, respectively. We further define
\begin{align*}
    L^p_\mathrm{loc}(I;X) 
    &:=
    \big\{ 
        u:I\to X \,\big\vert\, u \in L^p(J;X) \;\text{for every compact interval $J\subset I$}
    \big\}
    \\[1ex]
    L^p_\mathrm{uloc}(I;X) 
    &:=
    \left\{ u:I\to X \,\middle|\,
    \begin{aligned}
    &u \in L^p_\mathrm{loc}(I;X) \;\text{and}\; \exists C>0\; \forall t\in\R:\\
    &\|u\|_{L^p(I\cap[t,t+1);X)} \le C
    \end{aligned}
    \right\}.
\end{align*}
The spaces $W^{k,p}_\mathrm{loc}(I;X)$, $H^k_\mathrm{loc}(I;X)$, $W^{k,p}_\mathrm{uloc}(I;X)$, $H^k_\mathrm{uloc}(I;X)$ are defined analogously.

\item \textbf{Spaces of continuous functions.}
For any interval $I\subset\R$ and any Banach space $X$, $C(I;X)$ denotes the space of continuous functions mapping from $I$ to $X$ and $BC(I;X)$ denotes the space of functions in $C(I;X)$, which are additionally bounded. Moreover, $C_\mathrm{w}(I;X)$ denotes the space of functions mapping from $I$ to $X$, which are continuous on $I$ with respect to the weak topology on $X$, and $BC_\mathrm{w}(I;X)$ denotes the space of functions in $C_\mathrm{w}(I;X)$, which are additionally bounded.

\item \textbf{Spaces of functions with zero mean.}
For any $f\in H^1(\Omega)'$, its generalized spatial mean is defined as
\begin{equation*}
    \overline{f}:= |\Omega |^{-1} \langle f,1 \rangle,
\end{equation*}%
where $|\Omega |$ stands for the $n$-dimensional Lebesgue measure of
$\Omega$. 
Using this definition, we introduce the following function spaces:
\begin{align*}
    H^{-1}_{(0)}(\Omega) &:= \big\{ u\in H^1(\Omega)' \,:\, \ov u = 0 \big\} \subset H^1(\Omega)',\\
    L^2_{(0)}(\Omega) &:= \big\{ u\in L^2(\Omega) \,:\, \ov u = 0 \big\} \subset L^2(\Omega),\\
    H^1_{(0)}(\Omega) &:= \big\{ u\in H^1(\Omega) \,:\, \ov u = 0 \big\} \subset H^1(\Omega).
\end{align*}
As closed linear subspaces of the respective Hilbert space, these spaces are also Hilbert spaces.

\item \textbf{Spaces of divergence-free functions.}
If $\Omega$ is a bounded domain, we define the closed linear subspaces
\begin{align*}
    \LL^2_\sigma(\Omega)
    &:=\overline{\{\mathbf{u}\in \mathbf{C}^\infty_0(\Omega) \,\big\vert\, \operatorname{div}\ \mathbf{u}=0\}}^{\mathbf{L}^2(\Omega)}
    \subset \LL^2(\Omega), 
    \\
    \Vsigma
    &:=\overline{\{\mathbf{u}\in \mathbf{C}^\infty_0(\Omega) \,\big\vert\, \operatorname{div}\ \mathbf{u}=0\}}^{\mathbf{H}^1(\Omega)}
    \subset \HH^1(\Omega).
\end{align*}
In the case $\Omega =\TTn$, the corresponding closed linear subspaces are defined as
\begin{align*}
    \LL^2_\sigma(\Omega)
    &:=\overline{\big\{\mathbf{u}\in \mathbf{C}^\infty(\Omega) \,\big\vert\, 
    \Div\mathbf{u}=0 \;\text{and}\; \ov{\mathbf{u}} = 0
    \big\}}^{\mathbf{L}^2(\Omega)} \subset \LL^2(\Omega), 
    \\
    \Vsigma
    &:=\overline{\big\{\mathbf{u}\in \mathbf{C}^\infty(\Omega) \,\big\vert\, 
    \Div\mathbf{u}=0 \;\text{and}\; \ov{\mathbf{u}} = 0
    \big\}}^{\mathbf{H}^1(\Omega)} \subset \HH^1(\Omega).
\end{align*}
In both cases, Korn's inequality yields
\begin{equation}
\Vert \mathbf{u}\Vert \leq \sqrt{2}\Vert D\mathbf{u}%
\Vert\leq \sqrt{2}\Vert \Grad \mathbf{u}\Vert
\quad \text{ for all } \mathbf{u}\in \Vsigma.
\label{korn}
\end{equation}
Hence, $\|\Grad\cdot\|$ is a norm on $\Vsigma$ that is equivalent to the standard norm $\|\cdot\|_{\mathbf{H}^1(\Omega)}$.
\end{enumerate}  

\subsection{Assumptions}

The following general assumptions are supposed to hold throughout this paper.

\begin{enumerate}[label=\textnormal{(A\arabic*)},leftmargin=*]
    \item \label{ASS:GEN} 
    For $n\in\{2,3\}$, we either choose $\Omega$ to be a bounded domain in $\R^n$ of class $C^3$ or we take $\Omega$ to be the torus
    \begin{equation*}
        \TTn := \big[ \R / \big((2\mathbb Z + 1)\pi\big) \big]^n.
    \end{equation*}
    \item \label{ASS:RHONU}
    The density $\rho$, the viscosity $\nu$ and the mobility $m$ are positive constants. For convenience, we set $\rho=\nu=m=1$. This does not mean any loss of generality as the explicit choice of these positive constants does not have any impact on the mathematical analysis. 
    \item \label{ASS:JEPS} 
    Let $\Omega$ be given as in \ref{ASS:GEN}.
    For any $\eps>0$, let $\rho_\eps \in L^1\big(\R;[0,\infty)\big)$ be a given function satisfying the conditions
    \begin{align*}
		&\int_{0}^\infty\rho_\eps(r)\:r^{n-1}\:\text{d}r = \frac{2}{C_n},
        \quad\text{where}\quad
        C_n := \int_{\mathbb{S}^{n-1}}|\sigma_1|^2\:\text{d}\mathcal{H}^{n-1}(\sigma),
        \\
		&\lim\limits_{\eps\searrow 0}\int_{\delta}^\infty\rho_\eps(r)\:r^{n-1}\:\text{d}r = 0\;\;\;\text{for all }\delta>0.
	\end{align*}
    If $\Omega=\TTn$, we further demand that for all $\eps>0$, $\rho_\eps$ is compactly supported in $[0,\pi)$.
 
    For $X=\R^n$ if $\Omega$ is a bounded domain or $X=\TTn$ if $\Omega=\TTn$, we define
    \begin{equation*}
        J_\eps: X\rightarrow[0,\infty),
        \quad
        J_\eps(x) = \frac{\rho_\eps(|x|)}{|x|^2}
        \quad\text{for all $x\in X$ and all $\eps>0$},
    \end{equation*}
    and we additionally assume that $\rho_\eps$ is designed in such a way that $J_\eps \in W^{1,1}(X)$ (see, for instance, \cite{abels2023strong}).
\end{enumerate}

\medskip

For the singular potential in the free energy functional, we make the following assumptions, which not necessarily need to hold at the same time. We will specify further which of these assumptions are actually are needed in each stated result.
\begin{enumerate}[label=\textnormal{(S\arabic*)},leftmargin=*]
    \item \label{ASS:S1} 
    The potential $f:[-1,1]\to \R$ exhibits the decomposition
    \begin{equation*}
        f(s)=F(s)-\frac{\theta_0}{2}s^2 \quad\text{for all $s\in [-1,1]$}
    \end{equation*}
    with a given constant $\theta_0>0$. 
    Here, $F\in C([-1,1])\cap C^{2}(-1,1)$ has the properties 
    \begin{equation*}
    \lim_{r\rightarrow -1}F^{\prime }(r)=-\infty ,
    \quad \lim_{r\rightarrow 1}F^{\prime }(r)=+\infty ,
    \quad F^{\prime \prime }(s)\geq {\theta},
    \quad F'(0)=0
    \end{equation*}
    for all $s\in (-1,1)$ and a prescribed constant $\theta\in(0,\theta_0)$.
    Without loss of generality, we further assume $F(0)=0$. 
    In particular, this means that $F(s)\geq 0$ for all $s\in [-1,1]$.

    For convenience, we extend $f$ and $F$ onto $\R\setminus[-1,1]$ by defining 
    $f(s):=+\infty $ and $F(s):=+\infty $ for all $s\in\R\setminus [-1,1]$. 
    \item \label{ASS:S2} In addition to \ref{ASS:S1}, there exists $\beta>\frac12$ such that
    \begin{equation}
    \frac{1}{F^{\prime }(1-2\delta )}=O\left( \frac{1}{|\ln (\delta )|^{\beta }}%
    \right) ,\quad\text{ }\dfrac{1}{|F^{\prime }(-1+2\delta )|}=O\left( \frac{1}{%
    |\ln (\delta )|^{\beta }}\right) .  \label{est}
    \end{equation}
    as $\delta\to 0^+$.
    \item\label{ASS:S3} In addition to \ref{ASS:S1}, it holds
    \begin{alignat}{2}
    	\frac{1}{F^{\prime}(1-2\delta)}&=O\left(\frac{1}{\vert\ln(\delta)\vert}\right),
        &\quad\frac{1}{F^{\prime\prime}(1-2\delta)}&=O(\delta),
    	\label{F}
        \\
    	\dfrac{1}{\vert F^{\prime}(-1+2\delta)\vert }&=O\left(\frac{1}{\vert\ln(\delta)\vert}\right),
        &\quad\dfrac{1}{F^{\prime\prime}(-1+2\delta)}&=O\left(\delta\right).
    	\label{F2}
    \end{alignat}
    as $\delta\to 0^+$.
    Moreover, there exists $\gamma_{0}>0$ such that $F^{\prime \prime }$ is monotonously increasing on $(-1,-1+\gamma_0]$ and on $[1-\gamma _{0},1)$.
\end{enumerate}

\medskip

\begin{remark}\label{REM:LOG}
We point out that the logarithmic potential (also known as the \textit{Flory--Huggins potential}), which is given by
\begin{equation}
    f_\mathrm{log}(s)=F_\mathrm{log}(s)-\frac{\theta_0}{2}s^2 \quad\text{for all $s\in [-1,1]$}
    \label{f:LOG}
\end{equation}
with $F_\mathrm{log}(\pm 1) = \theta\ln(2)$ and
\begin{equation}
    F_\mathrm{log}(s)=\frac{\theta}{2}((1+s)\text{ln}(1+s)+(1-s)\text{ln}(1-s))
    \quad\text{for all $s\in(-1,1)$},
\label{F:LOG}
\end{equation}
satisfies all assumptions \ref{ASS:S1}--\ref{ASS:S3}. However, the assumptions \ref{ASS:S1}--\ref{ASS:S3} allow for a much more general class of potentials (see, e.g., \cite{GP} for a discussion).
\end{remark}

\comm{
{\color{red}[POSITIVE FACT]: I just realized that most of the estimates, and in particular the final result of convergence should work also for bounded domains! Indeed, with bounded domains we only lose controls on $\Grad \ceps$, which is not needed indeed!}
}

\subsection{Preliminaries}

\begin{enumerate}[label=\textnormal{(P\arabic*)},leftmargin=*]
\item\label{PRE:LAP} \textbf{The Laplace operator and its inverse.}
It is well-known that the operator
\begin{equation*}
    \mathcal{A} :H_{(0)}^{1}(\Omega )\rightarrow H_{(0)}^{-1}(\Omega ) ,
    \quad
    \left\langle \mathcal{A}  u,v\right\rangle_{H^1_{(0)}(\Omega)} =
    (\Grad u,\Grad v)
    \quad\text{for all $v\in H_{(0)}^{1}(\Omega )$}
\end{equation*}
is a continuous linear isomorphism. If $\Omega$ is a bounded domain, $\mathcal{A}$ can be interpreted as the negative Laplace operator with homogeneous Neumann boundary condition, and if $\Omega$ is the torus $\TTn$, $\mathcal{A}$ represents the Laplace operator with periodic boundary conditions.
We denote the inverse of $\mathcal{A}$, which is a bounded linear operator, by 
\begin{equation*}
    \mathcal{N}=\mathcal{A} ^{-1}:H_{(0)}^{-1}(\Omega )\rightarrow H_{(0)}^{1}(\Omega).
\end{equation*}
For any $g,h\in H_{(0)}^{-1}(\Omega )$, we set
\begin{equation*}
    (g,h)_{\ast} := \big(\Grad \mathcal{N}g,\Grad \mathcal{N}h\big),
    \quad
    \Vert g\Vert_{\ast }:=\Vert \Grad \mathcal{N}g\Vert.
\end{equation*}
This defines a bilinear form $(\cdot,\cdot)_{\ast}$ which is an inner product on the Hilbert space $H_{(0)}^{-1}(\Omega )$. Its induced norm $\Vert \cdot\Vert_{\ast}$ is equivalent to the standard operator norm on this space.

Moreover, due to elliptic regularity theory, there exists a constant $C>0$ such that for all $g\in L_{(0)}^{2}(\Omega)$,
\begin{equation}
\Vert \mathcal{N}g\Vert_{H^{2}(\Omega )}\leq C\Vert g\Vert.
\label{H_2}
\end{equation}

We further point out that the mapping 
$g\mapsto \big(\Vert g-\overline{g}
\Vert _{\ast }^{2}+|\overline{g}|^{2}\big)^{\frac{1}{2}}$ defines a norm $H^{1}(\Omega )^{\prime }$ that is equivalent to the standard operator norm on this space. 
\item\label{PRE:STOKES} \textbf{The Stokes operator and its inverse.}
The \textit{Stokes operator}, which is defined as
\begin{equation*}
    A_S: \Vsigma \to \Vsigma', \quad \mathbf{u}\mapsto (\Grad \mathbf{u},\Grad\vv)
    \quad\text{for all $\vv\in \Vsigma$}
\end{equation*}
is a continuous linear isomorphism.
For any $\vv, \mathbf{w} \in \Vsigma$, we set
\begin{equation*}
    (\vv,\mathbf{w})_{\sigma} := \big(\Grad A_S^{-1}\vv,\Grad A_S^{-1}\mathbf{w}\big),
    \quad
    \Vert \vv \Vert_{\sigma}:=\Vert \Grad A_S^{-1}\vv\Vert.
\end{equation*}
This defines a bilinear form $(\cdot,\cdot)_{\sigma}$, which is an inner product on the Hilbert space $\Vsigma'$.
Its induced norm $\norm{\cdot}_{\sigma}$ is equivalent to the standard operator norm on this space. In particular, due to Poincar\'e's inequality, there exists a constant $C_{S,1}>0$ such that for all $\mathbf{u}\in \Vsigma'$, it holds
\begin{align}
    \label{Stokes1}
    \norm{A_S^{-1} \mathbf{u}}_{\HH^1(\Omega)} \le C_{S,1} \norm{\mathbf{u}}_\sigma.
\end{align}
Moreover, due to regularity theory for the Stokes operator, there exists a constant $C_{S,2}>0$ such that for all $\mathbf{u}\in \LL^2_\sigma(\Omega)$, it holds
\begin{equation}
\Vert A_S^{-1}\mathbf{u}\Vert _{\HH^{2}(\Omega )}\leq C_{S,2} \Vert \mathbf{u}\Vert .
\label{Stokes2}
\end{equation}
In particular, using the \textit{Leray--Helmholtz projector} $\mathbf{P}_\sigma: \mathbf{L}^2(\Omega) \to \LL^2_\sigma(\Omega)$, we obtain the representation
\begin{equation*}
    A_S\big\vert_{\Vsigma\cap \mathbf{H}^2(\Omega)} = - \mathbf{P}_\sigma \Lap,
\end{equation*}
with $\Lap$ being the standard Laplace operator.
\end{enumerate}

\subsection{Known results and important tools}

In this section we collect some important results, which will play a crucial role in our subsequent analysis.

\subsubsection{Energy estimates}

In the following lemma, we provide some energy estimates that will be used frequently in our mathematical analysis.

\begin{lemma}\label{LEM:EN}
Let $\eps>0$ and let $J_\eps$ satisfy assumption \ref{ASS:JEPS}.
We use the notation
\begin{equation*}
    \mathcal{F}_\eps
    (u,v) := \frac14\intO\intO J_{\eps}(x-y)\big| u(x) -  v(y)\big|^2\:\textup{d}y\textup{d}x
\end{equation*}
and in accordance with \eqref{DEF:EN:CH}, we set
\begin{equation*}
    \mathcal{E}_\eps(u):=\mathcal{F}_\eps(u,u),
\end{equation*}
Then, the following estimates hold.
\begin{enumerate}[label=\textnormal{(\alph*)},topsep=0em, partopsep=0em,leftmargin=*]
    \item\label{EN:IEQ:1} For every $\gamma>0$, there exist constants $C_\gamma>0$ and $\eps_\gamma>0$ such that  
    \begin{align}
        \|u\|_{H^1(\Omega)}^2 &\leq \mathcal{E}_\eps(\Grad u) + C_\gamma\|u\|^2.
    \label{Poinc}
    \end{align}
    for all $\eps\in(0,\eps_\gamma]$ and all $u\in H^1(\Omega)$.
    \item\label{EN:IEQ:2} For every $\gamma>0$, there exist constants $C_\gamma>0$ and $\eps_\gamma>0$ such that  
    \begin{align}
        \|u\|^2 
        &\leq \gamma\mathcal{E}_\eps(u)  + C_\gamma\|u\|_{\ast}^2.
    \label{Poinc3}
    \end{align}
    for all $\eps\in(0,\eps_\gamma]$ and all $u\in L^2(\Omega)$.
\end{enumerate}
\end{lemma}

\medskip

\noindent For a proof of this lemma we refer to \cite[Lemma C.3]{EJ}. 

\subsubsection{Nonlocal to local convergence for the operator \texorpdfstring{$\mathcal{L}_\eps$}{}}
The nonlocal-to-local convergence $\mathcal L_\eps \to \mathcal L_0 = -\Lap$ as $\eps\to 0$ along with certain convergence rates has already been investigated in \cite{abels2023strong}, either for $\Omega = \R^n$ or for $\Omega$ being a bounded domain. Therefore, the following results are already known or can easily be obtained from those in \cite{abels2023strong}. 

\pagebreak[3]

\begin{proposition} \label{PROP:NLTL}
Suppose that \ref{ASS:GEN}--\ref{ASS:JEPS} hold, and for $\eps\ge 0$, let $\mathcal{L}_\eps$ be given by \eqref{DEF:NL:INTRO}.
\begin{enumerate}[label=\textnormal{(\alph*)},leftmargin=*]
    \item\label{PROP:NLTL:A} If $\Omega\subset\mathbb{R}^n$ is a bounded domain with $C^3$-boundary, 
    there exists a constant $K>0$ such that for all
    $c\in H^3(\Omega)$ with $\partial_{\mathbf{n}}c = 0$ on $\partial\Omega$ and all $\eps>0$, it holds
	\begin{align*}
			\big\|\mathcal{L}_\eps c - \mathcal{L}_0 c\big\| \leq K\sqrt{\eps}\|c\|_{H^3(\Omega)}.
	\end{align*}
    \item\label{PROP:NLTL:B} If $\Omega = \mathbb{T}^n$, there exists a constant $K>0$ such that for all $c\in H^3(\Omega)$ and all $\eps>0$, it holds
    \begin{align*}
        \big\|\mathcal{L}_\eps c - \mathcal{L}_0 c\big\|
        \leq K\eps\|c\|_{H^3(\Omega)}.
    \end{align*}
    \end{enumerate}
\end{proposition}

\begin{proof}
Part \ref{PROP:NLTL:A} has already been proven in \cite[Theorem~4.1 and Corollary~4.2]{abels2023strong}. Part \ref{PROP:NLTL:B} can be established similarly as \cite[Lemma~3.1]{abels2023strong}, which is the corresponding result for $\Omega=\R^n$, by using Fourier series instead of Fourier transformation.
\end{proof}

\subsubsection{Existence and uniqueness of weak and strong solutions to the local Model~H}

For the local Model~H (i.e., system \eqref{ModelH} with $\eps=0$), there already exists an extensive literature. In the following proposition, we collect the most important results concerning weak and strong well-posedness as well as separation properties. 

\begin{proposition} \label{PROP:WP:LOCAL}
Suppose that the assumptions \ref{ASS:GEN}--\ref{ASS:JEPS} and \ref{ASS:S1} hold.
We prescribe initial data $v_{0}\in \LL_{\sigma }^{2}(\Omega )$ and $c_0\in L^{\infty }(\Omega )\cap H^1(\Omega)$
with $\Vert c_0\Vert_{L^\infty(\Omega)}\leq 1$ and $|\overline{c _{0}}|<1$. 
Then there exists a global weak solution 
$$(\vv,c,\mu):\Omega\times [0,\infty) \to \R^n\times\R\times\R$$
to \eqref{ModelH} with $\eps=0$ with the following properties:
\begin{enumerate}[label=\textnormal{(\roman*)}, topsep=0em, partopsep=0em,leftmargin=*]
\item\label{E1loc} For any $T>0$, it holds
\begin{equation}
\begin{cases}
\vv\in C_{\mathrm{w}}([0,T];\LL_{\sigma }^{2}(\Omega ))\cap
L^{2}(0,T;\Vsigma), \\
c \in L^{\infty }(\Omega \times (0,T))\cap L^{4}(0,T;H^{2}(\Omega ))\text{
with }|c |<1\ \text{a.e. in }\Omega_T, \\
\delt \vv \in L^{\frac4n%
}(0,T;\Vsigma'),\\
 \delt c \in
L^{2}(0,T;H^{1}(\Omega )^{\prime }), \\
\mu \in L^{2}(0,T;H^{1}(\Omega )).%
\end{cases}
\label{regg-weakloc}
\end{equation}
\item For any $T>0$, the triplet $(\vv,c,\mu)$ fulfills the equations \eqref{ModelH:NS}--\eqref{ModelH:CH2} with $\eps=0$ in the weak sense, whereas the initial conditions \eqref{ModelH:IC} are fulfilled a.e.~in $\Omega$.
If $\Omega$ is a bounded domain, it further holds $\vv=\mathbf{0}$ and $\partial_\mathbf{n}c=0$ a.e.~on $\Gamma_T$. 
\end{enumerate}  
If $n=2$, the weak solution is unique.

Now, we additionally assume $\vv_{0}\in \Vsigma$, $c_0\in H^2(\Omega)$ and $\mu_0:=-\Lap c_0 + f'(c_0) \in H^1(\Omega)$.
If $\Omega$ is a bounded domain, we further assume $\partial_\mathbf{n}c_0=0$ a.e.~on $\Gamma$.
Then, there exists a unique right-maximal strong solution 
$$(\vv,p,c,\mu):\Omega\times [0,T_{\star}) \to \R^n\times\R\times\R\times\R$$
of system \eqref{ModelH} with $\eps=0$. If $n=2$, it holds $T_{\star}=\infty$.
This strong solution has the following properties:
\begin{enumerate}[label=\textnormal{(\roman*)}, topsep=0em, partopsep=0em, leftmargin=*, start=3]
\item \label{K1loc} It holds
\begin{align}
\left\{
\begin{aligned}
&\vv\in BC([0,T_{\star});\Vsigma)
\cap L_{\mathrm{uloc}}^{2}([0,T_{\star} );\HH^2(\Omega )\cap \Vsigma)
\\
&\qquad
\cap H_{\mathrm{uloc}}^{1}([0,T_{\star} );\LL_{\sigma }^{2}(\Omega )), 
\\
&p \in L_{\mathrm{uloc}}^{2}([0,T_{\star} );H_{(0)}^{1}(\Omega )), 
\\
&c \in L^{\infty }(0,T_{\star} ;L^{\infty }(\Omega ))\cap BC_{\mathrm{w}}([0,T_{\star} );{W^{2,p}}(\Omega )),
\\
& |c(x,t)|<1 \;\;\text{for almost all $x\in \Omega$ and all $t\in [0,T_{\star})$},
\\
&\partial_{t}c \in L^{\infty }(0,T_{\star} ;H^{1}(\Omega )^{\prime })\cap
L^{2}_{\mathrm{uloc}}(0,T_{\star} ;H^1(\Omega )),
\\ 
&F^{\prime }(c )\in L^{\infty
}(0,T_{\star} ;L^p(\Omega )),
\\
&\mu \in L^\infty(0,T_{\star};H^{1}(\Omega ))\cap L_{\mathrm{uloc}%
}^{2}([0,T_{\star} );H^{3}(\Omega )).%
\end{aligned}
\right.
\label{reggloc}
\end{align}
for all $p\in[2,\infty)$ if $n=2$ and all $p\in[2,6]$ if $n=3$.

\item \label{K2loc} The quadruplet $(\vv,p,c,\mu)$ fulfills the equations
\eqref{ModelH:NS}--\eqref{ModelH:CH2} a.e.~in $\Omega\times [0,T_{\star})$
and the initial condition \eqref{ModelH:IC} a.e.~in $\Omega$. 
If $\Omega$ is a bounded domain, it further holds $\vv=\mathbf{0}$ and $\partial_{\mathbf{n}}c=\partial_{\mathbf n}\mu=0$ a.e.~on $\Gamma \times (0,T_{\star})$.
\item \label{K3loc} If $n=2$ and assumption \ref{ASS:S2} additionally holds, there
exists $\delta_\star>0$ such that the strict separation property
\begin{equation}  \label{SP-Starloc}
\sup_{t\in \lbrack 0,\infty)}\Vert c (t)\Vert _{L^{\infty }(\Omega
)}\leq 1-\delta_\star
\end{equation}%
is fulfilled.
In particular, this entails
\begin{align}
c\in L^\infty(0,T;H^3(\Omega)) \quad\text{for all $T>0$}.
\label{regularity_c}
\end{align}
If $n=3$ and $\Vert c_0\Vert_{L^\infty(\Omega)}\leq 1-\delta_0$ holds for some $\delta_0\in(0,1)$,
there exist $0<T_0<T_{\star}$  such that the strict separation property
\begin{equation}  \label{SP-Starloc2}
\sup_{t\in \lbrack 0,T_0]}\Vert c (t)\Vert _{L^{\infty }(\Omega
)}\leq 1-\frac{\delta_0}{2}
\end{equation}
is fulfilled.
In particular, this entails
\begin{align}
c\in L^\infty(0,T_0;H^3(\Omega)).
\label{regularity_c2}
\end{align}
\end{enumerate}
\end{proposition}

\medskip

\begin{remark}\label{REM:WP:LOCAL}
\begin{enumerate}[label=\textnormal{(\alph*)},leftmargin=*]
    \item\label{REM:WP:LOCAL:SEP}
    To obtain the strict separation properties \eqref{SP-Starloc} and \eqref{SP-Starloc2} on an interval including the initial time, it is crucial that the initial datum $c_0$ is already strictly separated (i.e., $\Vert c_0\Vert_{L^\infty(\Omega)}\leq 1-\delta_0$ for some $\delta_0\in(0,1)$). In the case $n=2$, this already follows from the assumption $\mu_0=-\Lap c_0 + f'(c_0) \in H^1(\Omega)$ by means of De Giorgi iterations as employed in \cite[Theorem 4.3]{GP}.
    
    In the case $n=3$, at least up to now, the separation property \eqref{SP-Starloc2} can merely be obtained on a local neighborhood of the initial time.  
    For this result, it is sufficient to assume that the potential $f$ satisfies \ref{ASS:S1}.       
    If $n=2$, assuming both \ref{ASS:S1} and \ref{ASS:S2}, even a strict separation property on the entire interval $[0,\infty)$ can be established.     
    The question, whether this property can also be proven for $n=3$ is a challenging open problem.
    As shown in \cite[Section 6.1.1]{GP}, a strict separation property on the entire right-maximal interval $[0,T_{\star})$ can also be obtained in the case $n=3$ if slightly more singular potentials $f$ than the Flory--Huggins potential (see~Remark \ref{REM:LOG}) are used.
    
    The strict separation properties \eqref{SP-Starloc} and \eqref{SP-Starloc2} will be an essential ingredient in the proof of Theorem~\ref{THM:NLTL}.   
    \item\label{REM:WP:LOCAL:SMALL}
    As pointed out in Proposition~\ref{PROP:WP:LOCAL}, the unique strong solution exists globally in time (i.e, $T_{\star}=\infty$) if $n=2$. In the case $n=3$, due to the involved Navier--Stokes equation, only local existence of the strong solution (i.e, $T_{\star}<\infty$) for general initial data is known so far. 
    However, if the initial data are sufficiently close to a stationary point (i.e., a minimizer of the total energy), the global existence of the strong solution can still be ensured.
    If we additionally assume $c_0$ to be strictly separated, up to reducing the size of some norms of the initial data, the strict separation property \eqref{SP-Starloc2} can also be established globally in time (see \cite[Theorem 6.4]{GP}). 
    \item\label{REM:WP:LOCAL:TORUS}
    We point out that most of the results in the literature concerning the local version of the Model~H consider the case of bounded domains. However, it is clear that these results can usually be transferred to the case $\Omega=\mathbb{T}^n$ by slightly adapting the arguments. 
    
    Note that, if $\Omega=\TTn$, the assumption $\vv_0\in \LL^2_\sigma(\Omega)$ already includes the condition 
    $\ov{\vv_0}=\mathbf{0}$. This then implies $\int_\Omega \vv(t)\dx=\int_\Omega\vv_0\dx=0$ (cf.~\eqref{avg}) and therefore, we may apply the inverse Stokes operator $A_S^{-1}$ (see~\ref{PRE:STOKES}) directly on $\vv(t)$ for every $t\ge 0$ for which the solution exists.
    
    However, the assumption $\vv_0\in \LL^2_\sigma(\Omega)$ really does not mean any loss of generality as we could simply consider the difference $\vv - \ov{\vv}$ instead of $\vv$, which would not have a major impact on our mathematical analysis.
\end{enumerate}
\end{remark}

\begin{proof}[Proof of Theorem~\ref{PROP:WP:LOCAL}]
    If $\Omega$ is a bounded domain, the existence of a weak solution was established in \cite[Theorem~1]{Abels2009}. In the case $n=2$, the uniqueness of the weak solution was shown, e.g., in \cite{GMT2019}.
    Concerning the assertions on strong well-posedness we refer to \cite[Sections 4-5]{GMT2019}. 
    Note that the compatibility condition $\mu_0=-\Lap c_0 + f'(c_0) \in H^1(\Omega)$ is crucial for obtaining strong solutions.

    If $\Omega$ is the torus $\TTn$, the same results can be obtained by adapting the arguments in the aforementioned literature to the periodic setting. For instance, in the case $n=2$, the existence of a unique global strong solution was established in \cite{Giorgini2021}.
    
    In the case $n=2$, the strict separation property \eqref{SP-Starloc} can be established by following the line of argument in \cite[Theorem 3.3]{GP}, which is based on De Giorgi iterations. A crucial ingredient in this proof is the estimate
    \begin{align}
    \sup_{t\geq0}\Vert F'(c(t))\Vert_{L^p(\Omega)}\leq C\sqrt{p},\quad \text{for all $p\in[2,\infty)$}.
    \label{Flpa}
    \end{align}
    It can be derived by means of a Gagliardo--Nirenberg type estimate, which can be found, e.g., in \cite[p.~479]{Trudinger}. For more details, we refer to the derivation of \eqref{FLp} below, which is obtained by similar computations.  
    Once \eqref{Flpa} is established, one can proceed as in \cite[Theorem 3.3]{GP} to deduce the strict separation property \eqref{SP-Starloc} in the case $n=2$.     
    In this context, we recall that, as pointed out in Remark~\ref{REM:WP:LOCAL}\ref{REM:WP:LOCAL:SEP}, the assumption $\mu_0=-\Lap c_0 + f'(c_0) \in H^1(\Omega)$ already entails that the initial datum $c_0$ is strictly separated.
    As \eqref{SP-Starloc} directly implies $f'(c)\in L^\infty(0,\infty;H^1(\Omega))$, we apply elliptic regularity theory to the equation $-\Delta c=\mu-f'(c)$ in $\Omega\times (0,\infty)$ to conclude \eqref{regularity_c}.
    
    In the case $n=3$, the strict separation property \eqref{SP-Starloc2} can be shown similarly as in \cite[Corollary 4.4]{MZ} by means of a continuity argument. In view of the regularities in \eqref{reggloc}, we deduce
    \begin{align*}
        \norm{c(t)-c(s)}_{\ast}
        \leq \left| \int_s^t\norm{\delt c}_{\ast}\,\text{d}\tau \right| 
        \leq C\vert t-s\vert
    \end{align*}
    for all $s,t\in[0,T_{\star})$.
    This entails $c\in C^{0,1}([0,T_{\star});H^1(\Omega)')$. Using once more \eqref{reggloc}, we infer via interpolation that
    \begin{equation*}
    \norm{c(t)-c(s)}_{L^\infty(\Omega)}
    \leq C\norm{c(t)-c(s)}_{\ast}^{\beta}\norm{c(t)-c(s)}_{H^2(\Omega)}^{1-\beta}
    \leq C\vert t-s\vert^\beta
    \end{equation*}
    holds for all $s,t\in[0,T_{\star})$ and some suitably chosen $\beta\in(0,1)$. Therefore, we thus have $c\in C^{0,\beta}([0,T_{\star});L^\infty(\Omega))$, and hence, if there exists $\delta_0>0$ such that $\norm{c_0}_{L^\infty(\Omega)}\leq 1-\delta_0$, then \eqref{SP-Starloc2} holds for $T_0=\min\{(\delta_0/2)^{1/\beta}, T_{\star}\}$.   
    As a direct consequence of \eqref{regularity_c2}, we have $F'(c)\in L^\infty(0,T_0;H^1(\Omega))$. Hence, by applying elliptic regularity theory to the equation $-\Delta c=\mu-f'(c)$ in $\Omega\times (0,T_0)$, we conclude \eqref{regularity_c2}. Thus, the proof is complete.
\end{proof}


\section{Main results} 
\label{SECT:MAIN}

We are now ready to state the main results of the present paper.

\subsection{Existence and uniqueness of weak and strong solutions to the nonlocal Model~H}

This subsection is concerned with the existence and uniqueness of solutions to the nonlocal Model~H (i.e., system \eqref{ModelH} with $\eps>0$).

For $n=2,3$ and any fixed $\eps>0$, the weak well-posedness of the nonlocal Model~H has already been established in \cite[Theorem 1]{Frigeri2012a}. Furthermore, in the case $n=2$, the strong well-posedness theory of this Model~H has been developed in \cite{GGG2017} and, more in details, in \cite[Theorem 1.5, Theorem 1.9]{AGGP}, which even deals with the more general case of unmatched densities (i.e., $\rho$ is not constant and depends on the phase-field). Again, all these results are obtained for bounded domains, but as in the local case, they can be easily adapted to the case $\Omega=\TTn$.

In our first main result Theorem~\ref{THM:WP:NONLOCAL}, we show the existence of a unique local-in-time strong solution to the nonlocal Model~H (for any $\eps>0$) also in the case $n=3$. Moreover, for $n=2,3$, we are able to bound weak and strong solutions in suitable norms by a constant independent of $\eps$, at least provided that the considered $\eps$ is sufficiently small. 
These uniform estimates will be in essential ingredient in the nonlocal-to-local convergence of the Model~H.

\begin{theorem}\label{THM:WP:NONLOCAL}
Let $\eps>0$ and suppose that the assumptions \ref{ASS:GEN}--\ref{ASS:JEPS} and \ref{ASS:S1} hold.  
We prescribe initial data $\vv_{\eps,0}\in \LL_{\sigma }^{2}(\Omega )$  and $c _{\eps,0}\in L^{\infty }(\Omega )$ with $|\overline{c_{\eps,0}}|<1$.
We further assume that there exists a constant $C_0>0$ independent of $\eps$ such that 
\begin{equation}
    \label{assumptionE}
    E_\eps(\vv_{\eps,0},c_{\eps,0})\leq C_0,
\end{equation}
where $E_\eps$ is the energy functional defined in \eqref{DEF:EN}.
Then there exists a global weak solution 
$$(\vveps,\ceps,\mueps):\Omega\times [0,\infty) \to \R^n\times\R\times\R$$
to the nonlocal Model~H \eqref{ModelH} associated with $\eps$, and for any $T>0$, the following properties hold.
\begin{enumerate}[label=\textnormal{(\roman*)}, topsep=0em, partopsep=0em, leftmargin=*]
\item \label{E1} It holds
\begin{align}
&\begin{cases}
\vveps\in C_{\mathrm{w}}([0,T];\LL_{\sigma }^{2}(\Omega ))\cap
L^{2}(0,T;\Vsigma), 
\\
\delt \vveps \in L^{\frac4n%
}(0,T;(\Vsigma)^{\prime }),
\\
\ceps \in L^{\infty}(0,T;L^\infty(\Omega)) 
\text{ with }|\ceps |<1\ \text{a.e. in }\Omega \times (0,T), 
\\
\ceps \in L^2(0,T;H^1(\Omega)) \quad\text{if $\Omega=\TTn$,}
\\
\delt \ceps \in
L^{2}(0,T;H^{1}(\Omega )^{\prime }), 
\\
\mueps \in L^{2}(0,T;H^{1}(\Omega )).
\end{cases}
\label{regg-weak}
\end{align}%
\item \label{E2} The triplet $(\vveps,\ceps,\mueps)$ fulfills the equations \eqref{ModelH:NS}--\eqref{ModelH:CH2} 
in weak sense and the initial conditions $\vveps(\cdot ,0)=\mathbf{%
v}_{\eps,0}$ and $\ceps (\cdot ,0)=c_{\eps,0}$ hold in $\Omega $.
\item \label{E3}
There exists a constant $C_1(T)>0$ such that
\begin{align}
    \label{regg-weak-est1}
    \begin{aligned}
    &\|\vveps\|_{L^{2}(0,T;\Vsigma)}
    + \|\delt \vveps\|_{L^{4/n}(0,T;\Vsigma')}
    \\
    &\quad + \|\ceps \|_{L^{\infty }(\Omega \times (0,T))}
    + \|\delt \ceps \|_{L^2(0,T;H^1(\Omega)')}
    + \|\mueps \|_{L^2(0,T;H^1(\Omega))}
    \le C_1(T).
    \end{aligned}
\end{align}
There further exist $\eps_{w}=\eps_{w}(\theta_0)>0$ and a constant $C_2(T)>0$ such that 
\begin{align}
    \label{regg-weak-est2}
    \|\ceps\|_{L^2(0,T;H^1(\Omega))} \le C_2(T)
    \quad\text{if $\Omega=\TTn$ and $\eps\in(0,\eps_w]$.}
\end{align}
\end{enumerate} 

Now, we additionally assume $\vv_{\eps,0}\in \Vsigma$, $c _{\eps,0}\in
H^{1}(\Omega )$, with $|\overline{c _{\eps,0}}|<1$, $F^{\prime }(c _{\eps, 0})\in
L^{2}(\Omega )$ and $F^{\prime \prime }(c _{\eps,0})\Grad c _{\eps,0}\in
\LL^{2}(\Omega)$. We further demand that there exists a constant $C_0>0$ independent of $\eps$ such that 
\begin{align}
\norm{D\vv_{\eps,0}}+\norm{\Grad\mu_{\eps,0}}\leq C_0,
\label{assumption0}
\end{align}
where $\mu_{\eps,0}:=\mathcal{L}_\eps c_{\eps,0}+f'(c_{\eps,0})$.
Then, there exists a unique right-maximal strong solution 
$$(\vveps,p_\eps,\ceps,\mueps):\Omega\times [0,T_{\eps,*}) \to \R^n\times\R\times\R\times\R$$
of system \eqref{ModelH} associated with $\eps$.
If $n=2$ it holds $T_{\eps,*}=\infty$.
This strong solution has the following properties:
\begin{enumerate}[label=\textnormal{(\roman*)}, topsep=0em, partopsep=0em, leftmargin=*, start=4]
\item \label{K1} For any $T\in(0,T_{\eps,*})$, it holds 
\begin{equation}
\begin{cases}
\vveps\in BC([0,T];\Vsigma)\cap L^{2}(0,T;\HH^2(\Omega )\cap \Vsigma)\cap H^{1}(0,T ;\LL_{\sigma }^{2}(\Omega )), 
\\
\peps\in L^{2}(0,T ;H_{(0)}^{1}(\Omega )), 
\\
\ceps \in L^\infty\big(0,T;H^1(\Omega) \cap L^\infty(\Omega)\big) 
\text{ with }|\ceps |<1\ \text{a.e. in }\Omega \times (0,T), 
\\
\delt \ceps \in L^{\infty }(0,T ;H^{1}(\Omega )^{\prime })\cap
L^{2}(0,T ;L^{2}(\Omega )),
\\
F^{\prime }(\ceps )\in L^2(0,T;H^1(\Omega)) \cap L^{\infty}(0,T ;L^{p}(\Omega )),
\\
\mueps \in L^\infty(0,T ;H^{1}(\Omega ))\cap L
^{2}(0,T;H^{2}(\Omega )).%
\end{cases}
\label{regg}
\end{equation} 
for all $p\in[2,\infty)$ if $n=2$ and all $p\in[2,6]$ if $n=3$. 
\item \label{K2} $(\vveps,\peps,\ceps, \mueps )$ fulfills the equations
\eqref{ModelH:NS}--\eqref{ModelH:CH2}
a.e.~in $\Omega\times [0,T_{\eps,*})$
and the initial condition \eqref{ModelH:IC} a.e.~in $\Omega$. 
If $\Omega$ is a bounded domain, it further holds $\vveps=\mathbf{0}$ and $\partial_{\mathbf n}\mu_\eps=0$ a.e.~on $\Gamma \times (0,T_{\eps,*})$.
\item \label{K3}
If $n=3$, there exist $\eps_s\in (0,\eps_w]$ and $T_*\in (0,T_{\eps,*})$ independent of $\eps$ 
such that for any $T\in(0,T_*]$, there exist constants $C_3(T),C_4(T)>0$ such that
\begin{align}
    \label{regg-est1}
    \begin{aligned}
    &\|\vveps\|_{L^{\infty}(0,T;\Vsigma)}
    + \|\vveps\|_{L^{2}(0,T;\HH^2(\Omega ))}
    + \|\vveps\|_{H^{1}(0,T ;\LL_{\sigma }^{2}(\Omega))} 
    \\
    &\quad
    + \|\peps\|_{L^2(0,T;H^1(\Omega))}
    + \|\delt \ceps \|_{L^\infty(0,T;H^1(\Omega)')}
    + \|\delt \ceps \|_{L^2(0,T;L^2(\Omega))}
    \\
    &\quad
    + \|F'(\ceps) \|_{L^\infty(0,T;L^p(\Omega))}
    + \|\mueps \|_{L^\infty(0,T;H^1(\Omega))}
    \le C_3(T)
    \qquad\text{if $\eps\in(0,\eps_s]$,}
    \end{aligned}
\end{align}
for all $p\in[2,6]$, and
\begin{equation}
    \label{regg-est2}
    \|\ceps \|_{L^\infty(0,T;H^1(\Omega))}
        + \|\mueps \|_{L^2(0,T;H^2(\Omega))} 
    \le C_4(T) \quad\text{if $\Omega=\TTn$ and $\eps\in(0,\eps_s]$.}
\end{equation}

If $n=2$, there exists $\eps_s\in (0,\eps_w]$ such that \eqref{regg-est1} and \eqref{regg-est2} even hold for every $T>0$ and every $p\in[2,\infty)$.

\item \label{K4} If $n=2$, we now additionally assume that \ref{ASS:S2} holds, and if $n=3$, we additionally assume that \ref{ASS:S3} is fulfilled. Then, for any $\tau >0$, there exists $\delta_{\eps,\tau} \in (0,1)$ such that the strict separation property
\begin{equation}
\sup_{t\in \lbrack \tau ,T_{\eps,*})}\Vert \ceps (t)\Vert _{L^{\infty }(\Omega
)}\leq 1-\delta_{\eps,\tau}  \label{delt}
\end{equation}
holds.
Moreover, if we further assume that 
$\Vert c _{\eps,0}\Vert _{L^{\infty}(\Omega )}\leq 1-\delta _{\eps,0}$ for some $\delta _{\eps,0}\in (0,1)$, then there exists $\delta_{\eps,0}^{\ast}>0$ such that the strict separation property
\begin{equation}  \label{SP-Star}
\sup_{t\in \lbrack 0,T_{\eps,*})}\Vert \ceps (t)\Vert _{L^{\infty }(\Omega
)}\leq 1-\delta_\eps^{\ast}
\end{equation}
holds.
In this case, we further have $\delt \mueps \in L^{2}(0,T
;L^{2}(\Omega ))$ for every $T\in (0,T_{\eps,*})$.
\end{enumerate}
We point out that the constants $C_1(T),...,C_4(T)$ may depend on the choice of $\Omega$, the initial data and the system parameters, but are independent of $\eps$.
\end{theorem}

\medskip

\begin{remark}
We remark that, in case $\Omega=\TTn$, the assumptions on the initial data for strong solutions already entail that 
\begin{align*}
    &\intO F''(c_{\eps,0})\vert \Grad c_{\eps,0}\vert^2 \dx +\frac14\intO \intO J_\eps(x-y)\vert \Grad c_{\eps,0}(x)-\Grad c_{\eps,0}(y)\vert^2 \dx \dy  
    \\
    &\quad \leq C(1+\Vert \Grad\mu_{\eps,0}\Vert^2)\leq C
\end{align*}
with a constant $C>0$ that does not depend on $\eps$, as long as $\eps$ is sufficiently small. In fact, this estimate can be shown similarly as estimate \eqref{essen}, which will be derived in the proof of Theorem~\ref{THM:WP:NONLOCAL}.
\end{remark}

\subsection{Nonlocal-to-local convergence for strong solutions of the Model~H}
As our second main result, which is stated in Theorem~\ref{THM:NLTL}, we establish the nonlocal-to-local convergence of the Model~H. More precisely, we show that for any suitable sequence of initial data, the strong solutions of the nonlocal Model~H with $\eps>0$ converge to a strong solution of the local Model~H as the parameter $\eps$ is sent to zero. This convergence is quantified by certain convergence rates.

\begin{theorem}\label{THM:NLTL}
Suppose that the assumptions \ref{ASS:GEN}--\ref{ASS:JEPS} and \ref{ASS:S1} hold, and if $n=2$, we further assume that \ref{ASS:S2} holds. If $\Omega$ is a bounded domain, we set $\alpha:=\tfrac 12$, and if $\Omega=\TTn$, we set $\alpha:=1$.

We prescribe initial data $\vv_0\in\Vsigma$ and $c_0\in H^2(\Omega)$ with $\|c_0\|_{L^\infty(\Omega)} \le 1$,  $|\ov{c_0}|<1$ and $-\Lap c_0 + f'(c_0) \in H^1(\Omega)$.
If $\Omega$ is a bounded domain, we additionally assume $\deln c_0 = 0$ a.e.~on $\Gamma$, and if $n=3$, we further assume that $\Vert c _{0}\Vert _{L^{\infty}(\Omega )}\leq 1-\delta _{0}$ for some $\delta _{0}\in (0,1)$. 
This ensures the existence of the corresponding unique right-maximal strong solution
$$(\vv,p,c,\mu):\Omega\times [0,T_{\star}) \to \R^n\times\R\times\R\times\R$$
to the local Model~H,
which satisfies the properties \ref{K1loc}--\ref{K3loc} of Proposition~\ref{PROP:WP:LOCAL}.

For any $\eps>0$, we prescribe initial data $\vv_{\eps,0}\in \LL_{\sigma }^{2}(\Omega )$  and $c _{\eps,0}\in L^{\infty }(\Omega )$ with $|\overline{c_{\eps,0}}|<1$, $F^{\prime }(c _{\eps, 0})\in L^{2}(\Omega )$ and $F^{\prime \prime }(c _{\eps,0})\Grad c _{\eps,0}\in
\LL^{2}(\Omega)$. 
We further demand that there exists constants $C_0,C_1,C_2>0$ independent of $\eps$ such that 
\begin{align}
    \label{initiald:0}
    E_\eps(\vv_{\eps,0},c_{\eps,0})
    &\leq C_0,
    \\
    \label{initiald:1}
    \norm{D\vv_{\eps,0}}+\norm{\Grad\mu_{\eps,0}}
    &\leq C_1,
    \\
    \label{initiald:2}
    \norm{\vv_{\eps,0}-\vv_0}_{\sigma}  + \norm{c_{\eps,0}-c_0-( \ov{c_{\eps,0}}-\ov{c_0})}_{\ast} + \abs{\ov{c_{\eps,0}}-\ov{c_0}}
    &\leq C_2\eps^\alpha,
\end{align}
where $\mu_{\eps,0}:=\mathcal{L}_\eps c_{\eps,0}+f'(c_{\eps,0})$.
This ensures the existence of the corresponding unique right-maximal strong solution
$$(\vveps,\peps,\ceps,\mueps):\Omega\times [0,T_{\eps,*}) \to \R^n\times\R\times\R\times\R$$
to the nonlocal Model~H associated with $\eps$,
which satisfies the properties \ref{K1} and \ref{K2} of Theorem~\ref{THM:WP:NONLOCAL}.%

We now choose $T_0>0$ as in Proposition~\ref{PROP:WP:LOCAL} and $T_\ast>0$ as in Theorem~\ref{THM:WP:NONLOCAL}, and we set $T_\diamond := \infty$ if $n=2$ and $T_\diamond := \min\{T_0,T_\ast\}$ if $n=3$.
Then, for any $T\in (0,T_\diamond)$, there exists a constant $C(T)>0$ independent of $\eps$ such that
\begin{align}
    \label{finale}
    \begin{aligned}
    &\norm{\vv_{\eps}-\vv}_{L^\infty(0,T;\Vsigma')}
    +\norm{c_{\eps}-c}_{L^\infty(0,T;H^1(\Omega)')}
    \\
    &\quad
    + \norm{\vveps-\vv}_{L^2(0,T;L^2(\Omega))}
    +\norm{\ceps-c}_{L^2(0,T;L^2(\Omega))}
    + \int_0^T \mathcal{E}_\eps(\ceps-c) \dt
    \leq C(T)\eps^\alpha
    \end{aligned}
\end{align}
for all $\eps \in (0,\eps_s]$, where $\eps_s$ is the number introduced in Theorem~\ref{THM:WP:NONLOCAL}.
\end{theorem}

\medskip

\begin{remark}
    \begin{enumerate}[label=\textnormal{(\alph*)},leftmargin=*]
    \item As the convergence rates are mainly inherited from Proposition~\ref{PROP:NLTL}, we obtain a higher convergence rate if $\Omega=\TTn$ than in the case of $\Omega$ being a bounded domain in $\R^n$.
    \item We point out that assuming a strictly separated initial datum $c_0$ in the case $n=3$ is necessary to prove the assertion, as the strict separation property \eqref{SP-Starloc2} is essential. 
    In the case $n=2$, however, the strict separation of the initial datum $c_0$ does not have to be imposed as an additional assumption (see also Remark~\ref{REM:WP:LOCAL}\ref{REM:WP:LOCAL:SEP}). 
    Moreover, it is worth mentioning that assuming strict separation of the initial data $\{c_{\eps,0}\}_{\eps>0}$ is not necessary, not even in three dimensions.
    \end{enumerate}
\end{remark}


\section{Proof of Theorem~\ref{THM:WP:NONLOCAL}} 
\label{PROOF:WELLPOSEDNESS}

\subsection{Existence of weak and strong solutions} 
\label{SUB:EX}

In the case $n=2$, under the respective assumptions made in Theorem~\ref{THM:WP:NONLOCAL}, the existence of a weak solution satisfying \ref{E1} and \ref{E2} has already been established in \cite{Frigeri2021}, and the existence of a strong solution satisfying \ref{K1} and \ref{K2} has been shown in \cite{GGG2017, AGGP}.
In fact, in \cite{Frigeri2021} and \cite{AGGP} even the more general case of unmatched densities is considered.

In the case $n=3$, the existence of a weak solution satisfying \ref{E1} and \ref{E2} has also been proven in \cite{Frigeri2021}.
The existence of a strong solution satisfying \ref{K1} and \ref{K2} can be shown by proceeding similarly as in \cite{AGGP}. More precisely, the uniform estimates that will be established in Subsection~\ref{SUB:EST:WEAK} and Subsection~\ref{SUB:EST:STRONG} can also be rigorously derived in the framework of a semi-Galerkin scheme as employed in \cite{AGGP}. This means that only the velocity field is discretized via a Galerkin ansatz, and the overall approximate solution is then constructed by means of a fixed point argument (as in \cite[Theorem 1.5]{AGGP}) relying on previous existence results for the convective nonlocal Cahn--Hilliard equation (see \cite[Theorem 2.2]{PS}).   

However, in contrast to the two-dimensional case, it cannot be shown that the constructed right-maximal strong solution exists for all times. This is, of course, due to the involved Navier--Stokes equation for which global existence of regular solutions in three dimensions is still an open problem. For the maximal existence time of strong solutions, a concrete lower bound $T_*$ that is uniform in $\eps$ will be explicitly derived in Subsection~\ref{SUB:EST:STRONG}.

\subsection{Uniqueness of the right-maximal strong solution}
\label{SUB:UNIQ}

In the case $n=2$, the proof of uniqueness of weak solutions to \eqref{ModelH} is quite standard, and we refer, for instance, to \cite[Theorem~6.2]{GGG2017}. The uniqueness of strong solutions to \eqref{ModelH} (even in the more general case of unmatched viscosities) has been shown in \cite[Theorem~1.9]{AGGP}. 
In the case $n=3$, the uniqueness of weak solutions is of course an open problem due to the involved Navier--Stokes equation. However, we are able to prove the uniqueness of the right-maximal strong solution.

Therefore, in the remainder of this subsection, we choose $n=3$, we fix an arbitrary $\eps>0$, and we set $T:=T_{\eps,*}$. As the choice of $\eps$ does not matter in this subsection, the index $\eps$ will simply be omitted.

Furthermore, in this subsection, the letter $C$ denotes generic positive constants that may depend on the choice of $\Omega$, the initial data and the system parameters including $\eps$.
The exact value of $C$ may vary in the subsequent line of argument.

We consider two sets of initial data $(\vv_{0,1},c_{0,1})$
and $(\vv_{0,2},c_{0,2})$ which satisfy the assumptions for the existence of strong solutions imposed in
Theorem \ref{THM:WP:NONLOCAL}. In addition, we assume that $\ov{c_{0,1}} = \ov{c_{0,2}}$.
For $i=1,2$, let $(\vv_i,p_i,c_i,\mu_i)$ denote a strong solution of \eqref{ModelH} associated with $\eps$ corresponding to the initial data $(\vv_{0,i},c_{0,i})$, respectively. We further write
\begin{align*}
    (\vv_{0},c_{0}) &:= (\vv_{0,1},c_{0,1}) - (\vv_{0,2},c_{0,2}),
    \\
    (\vv,p,c,\mu) &:= (\vv_1,p_1,c_1,\mu_1) - (\vv_2,p_2,c_2,\mu_2).
\end{align*}
This means that the quadruplet $(\vv,p,c,\mu)$ fulfills the following system of equations in the strong sense:
\begin{subequations}
\label{U:ModelH}
\begin{align}
    \label{U:ModelH:NS}
	&\delt \vv + (\vv_1 \cdot \Grad)\vv + (\vv \cdot \Grad)\vv_2 - \Lap\vv + \Grad p = \mu_1 \Grad c_1 - \mu_2\Grad c_2,
    \quad \Div(\vv) = 0
    &&\text{in}\;\Omega_T,
    \\
    \label{U:ModelH:CH1}
	&\delt c + \vv_1\cdot\Grad c + \vv\cdot\Grad c_2 = \Delta\mu
    \qquad&&\text{in}\;\Omega_T, 
    \\
    \label{U:ModelH:CH2}
	&\mu = \mathcal{L}_\eps c + F^{\prime }(c_{1})-F^{\prime }(c _{2})+\theta_0c
    \qquad&&\text{in}\;\Omega_T,
    \\
    \label{U:ModelH:IC}
    &\vv\vert_{t=0} = \vv_0, \quad c\vert_{t=0} = c_0
    &&\text{in}\; \Omega.
\end{align}
If $\Omega$ is a bounded domain, $(\vv,p,c,\mu)$ also satisfies the boundary conditions
\begin{align}
    \label{U:ModelH:BC}
    \vv = 0,\quad \deln \mu =0
    \quad\text{on}\; \Gamma_T.
\end{align}
\end{subequations}

Integrating \eqref{U:ModelH:CH1} over $\Omega$ and recalling that $\vv$ and $\vv_1$ are divergence-free, we first observe 
\begin{equation*}
    \ddt \intO c \dx 
    = \intO \delt c \dx
    = \intO \Div \big( \Grad \mu - \vv_1 c - \vv c_2 \big) \dx = 0
\end{equation*}
by means of Gau\ss's divergence theorem. This means that
\begin{equation}
    \ov{c}(t) = \ov{c_0} = 0 \quad\text{for all $t\in [0,T]$.}
\end{equation}
We now test \eqref{U:ModelH:NS} by $A_S^{-1}\vv$ (cf.~\ref{PRE:STOKES}) and
\eqref{U:ModelH:CH1} by $\mathcal{N}c$ (cf.~\ref{PRE:LAP}), and we add the resulting equations. Integrating by parts and invoking the identities
\begin{align*}
    \frac{1}{2} \ddt \norm{\vv}_{\sigma}^{2}
    &= \frac{1}{2} \ddt \Vert \Grad A_S^{-1} \vv \Vert^{2}
    = \langle \delt \vv , A_S^{-1} \vv \rangle_{\Vsigma},
    \\
    \frac{1}{2} \ddt \norm{c}_{*}^{2}
    &= \frac{1}{2} \ddt \Vert \Grad \mathcal{N}c \Vert^{2}
    = \langle \delt c , \mathcal{N}c \rangle_{H^1_{(0)}(\Omega)},
\end{align*}
we infer
\begin{equation}
\label{EQ:UNIQ:1}
    \begin{split}
    & \frac{1}{2} \ddt \Big(\norm{\vv}_\sigma^{2}+\norm{c} _{\ast}^{2}\Big) 
        +\norm{\vv}
        +(\mu ,c)
    \\
    &\quad = (\vv_{1}\otimes \vv,\Grad A_S^{-1}\vv)
        + (\vv\otimes \vv_{2},\Grad A_S^{-1}\vv) 
        - ( \vv_{1}\cdot \Grad c ,\mathcal{N}c )
    \\
    & \qquad 
        - ( \vv\cdot \Grad c _{2},\mathcal{N}c )
        + ( \mu_1 \Grad c_1 - \mu_2\Grad c_2 ,A_S^{-1}\vv ).
    \end{split}%
\end{equation}%
Replacing $\mu$ by means of \eqref{U:ModelH:CH2} and recalling the monotonicity of $F'$, the definition of 
$\mathcal{L}_\eps$ and the properties of $J_\eps$ (see~\ref{ASS:JEPS}), we use Young's inequality (both the version for products and the version for convolutions) to derive the estimate
\begin{align}
    \label{EQ:UNIQ:2}
    \begin{aligned}
    \left( \mu ,{c} \right) 
    &\geq \theta_0 \Vert {c} \Vert^{2}
        + (J_\eps\ast 1)\big(c,c\big) 
        - \big( J_\eps\ast c, c\big) 
    \\
    &\ge  \theta_0 \Vert {c} \Vert^{2}
        - \big( (\Grad J_\eps)\ast c, \Grad\mathcal{N}(c)\big) 
    \\
    &\ge \theta_0 \Vert {c} \Vert^{2}
        -\norm{J_\eps}_{W^{1,1}(X)} \, \norm{c} \, \norm{c}_\ast
    \\
    &\ge \frac{10}{16}\theta_0 \norm{c}^2 - C \norm{c}_*^2.
    \end{aligned}
\end{align}
Furthermore, recalling that the velocity fields $\vv_i$, ${i=1,2}$, are divergence-free, and using integration by parts, the Gagliardo--Nirenberg inequality, estimate \eqref{H_2} and Young's inequality, we deduce 
\begin{align}
    \label{EQ:UNIQ:3}
    \begin{aligned}
    \abs{( \vv_1 \cdot \Grad c,\mathcal{N}c )}
    &= \abs{(\vv_1 c,\Grad \mathcal{N}c)}
    \le \norm{\vv_1}_{\LL^6(\Omega)} \, \norm{c}\, \norm{\Grad\mathcal{N}c}_{\LL^3(\Omega)} 
    \\
    &\le C \norm{\vv_1}_{\LL^6(\Omega)} \, \norm{c}\, 
        \norm{\Grad\mathcal{N}c}_{\HH^1(\Omega)}^{1/2}
        \norm{\Grad\mathcal{N}c}^{1/2}
    \\
    &\le C \norm{\vv_1}_{\LL^6(\Omega)} \, \norm{c}^{3/2}\, 
        \norm{c}_{*}^{1/2}
    \\
    &\le \frac{\theta_0}{16}\norm{c}^2 + C \norm{\vv_1}_{\LL^6(\Omega)}^{4} \norm{c}_*^2.
    \end{aligned}
\end{align}
Proceeding similarly, we obtain
\begin{align}
    \label{EQ:UNIQ:4.1}
    \begin{aligned}
    \abs{(\vv_{1}\otimes \vv,\Grad A_S^{-1}\vv)}
    \le \norm{\vv_1}_{\LL^6(\Omega)} \, \norm{\vv} \, \norm{\Grad A_S^{-1}\vv}_{\LL^3(\Omega)}
    \le \frac{1}{8} \norm{\vv}^2 + C \norm{\vv_1}_{\LL^6(\Omega)}^{4} \norm{\vv}_\sigma^2,
    \end{aligned}
\end{align}
and analogously, we get
\begin{align}
    \label{EQ:UNIQ:4.2}
    \abs{(\vv\otimes \vv_2,\Grad A_S^{-1}\vv)}
    \le \frac{1}{8} \norm{\vv}^2 + C \norm{\vv_2}_{\LL^6(\Omega)}^{4} \norm{\vv}_\sigma^2. 
\end{align}
Recalling that $\abs{c}\le\abs{c_1}+\abs{c_2} <2$ a.e.~in $\Omega_T$ and that $\vv$ is divergence free, we further deduce
\begin{align}
    \label{EQ:UNIQ:5}
    \begin{aligned}
    \abs{( \vv\cdot \Grad c_{2},\mathcal{N}c )}
    &= \abs{(\vv c_2,\Grad \mathcal{N}c)}
    \le 2 \norm{\vv}_{\LL^2(\Omega)} \, \norm{\Grad\mathcal{N}c}
    \le \frac{1}{8} \norm{\vv}^2 + C \norm{c}_*^2.
    \end{aligned}
\end{align}
Furthermore, expressing $\mu_1$ and $\mu_2$ by means of \eqref{ModelH:CH2}, we deduce
\begin{align*}
    \begin{aligned}
    \big( \mu_1 \Grad c_1 - \mu_2\Grad c_2 , A_S^{-1}\vv \big) 
    &= \big( (J_\eps\ast 1) c_1 \, \Grad c_1 - (J_\eps\ast 1) c_2 \, \Grad c_2 ,A_S^{-1}\vv \big) 
    \\
    &\quad 
        + \big( (J_\eps\ast c_1) \, \Grad c_1 - (J_\eps\ast c_2) \, \Grad c_2 ,A_S^{-1}\vv \big)
    \\
    &\quad 
        + \big( \Grad F(c_1) - \Grad F(c_2) , A_S^{-1}\vv \big) 
    \\
    &\quad
        + \tfrac{1}{2}\theta_0 \big( \Grad(c_1^2) - \Grad(c_2^2) , A_S^{-1}\vv \big) .
    \end{aligned}
\end{align*}
As $A_S^{-1}\vv$ is divergence-free, the last two lines of the right-hand side vanish after integrating by parts. Moreover, reformulating the first two lines and using integration by parts, we obtain
\begin{align*}
    \begin{aligned}
    \big( \mu_1 \Grad c_1 - \mu_2\Grad c_2 , A_S^{-1}\vv \big) 
    &= \big( (J_\eps\ast 1) c_1 \, \Grad c + (J_\eps\ast 1) c \, \Grad c_2 ,A_S^{-1}\vv \big) 
    \\
    &\qquad 
        + \big( (J_\eps\ast c_1) \, \Grad c + (J_\eps\ast c) \, \Grad c_2 ,A_S^{-1}\vv \big) 
    \\[1ex]
    &= - \big( (\Grad J_\eps\ast 1) c_1 \, c , A_S^{-1}\vv \big) 
        - \big( (J_\eps\ast 1) c \, \Grad c , A_S^{-1}\vv \big) 
    \\
    &\qquad 
        + \big( (\Grad J_\eps\ast c_1) \, c , A_S^{-1}\vv \big) 
        - \big( (J_\eps\ast c) \, \Grad c_2 , A_S^{-1}\vv \big) 
    \\[1ex]
    &=: I_1 + I_2 + I_3 + I_4.
    \end{aligned}
\end{align*}
We now recall that $J_\eps \in W^{1,1}(X)$ (cf.~\ref{ASS:JEPS}) and that $\abs{c_1}<1$ a.e.~in $\Omega_T$. Invoking H\"older's inequality, Young's inequality (both for products and for convolutions) and Agmon's inequality along with the properties of the operator $A_S^{-1}$ (cf.~\ref{PRE:STOKES}), the terms $I_1,...,I_4$ can be estimated as follows:
\begin{align*}
    \abs{I_1} 
    &\le \norm{J_\eps}_{W^{1,1}(X)}\, \norm{c_1}_{L^\infty(\Omega)}\, 
        \norm{c}\, \norm{A_S^{-1}\vv} 
    \le \frac{\theta_0}{16}\norm{c}^{2} + C \norm{\vv} _{\sigma}^{2}\,,
    \\[1ex]
    \abs{I_2}
    &\le \norm{J_\eps}_{W^{1,1}(X)}\, \norm{c} \big(\norm{\Grad c_1} + \norm{\Grad c_2} \big) 
        \norm{A_S^{-1}\vv}_{\LL^\infty(\Omega)}
    \\
    &\le C \norm{c}\, \norm{A_S^{-1}\vv}_{\HH^2(\Omega)}^{1/2}\, \norm{A_S^{-1}\vv}_{\HH^1(\Omega)}^{1/2} 
    \le C \norm{c}\, \norm{\vv}^{1/2}\, \norm{\vv}_\sigma^{1/2}
    \\
    &
    \le \frac{\theta_0}{16}\Vert c\Vert^{2} + \frac{1}{4} \norm{\vv}^2 + C \norm{\vv}_\sigma^2\,,
    \\[1ex]
    \abs{I_3}
    &\le \norm{J_\eps}_{W^{1,1}(X)}\, \norm{c_1}_{L^\infty(\Omega)}\, 
        \norm{c} \norm{A_S^{-1}\vv}\,    
    \le \frac{\theta_0}{16}\Vert c\Vert ^{2}+C \norm{\vv} _{\sigma}^{2}\,,
    \\[1ex]
    \abs{I_4}
    &\le \norm{J_\eps}_{W^{1,1}(X)}\, \norm{c}\, \norm{\Grad c_2} \norm{A_S^{-1}\vv} 
    \le \frac{\theta_0}{16}\norm{c}^{2} + C \norm{\vv} _{\sigma}^{2}.
\end{align*}
In summary, we thus have
\begin{equation}
    \label{EQ:UNIQ:6}
    \bigabs{ \big( \mu_1 \Grad c_1 - \mu_2\Grad c_2 , A_S^{-1}\vv \big) }  
    \le \frac{\theta_0}{4} \norm{c}^2 + \frac{1}{8} \norm{\vv}^2
        + C \norm{\vv}_{\sigma}^2 .
\end{equation}
Combining \eqref{EQ:UNIQ:1}--\eqref{EQ:UNIQ:6}, we conclude
\begin{equation*}
    \begin{split}
    & \frac{1}{2} \ddt \Big(\norm{\vv}_\sigma^{2}+\norm{c} _{\ast}^{2}\Big) 
        + \frac{1}{2} \norm{\vv}
        + \frac{\theta_0}{2} \norm{c}
    \\
    &\quad \le 
        C \Big( 1 + \norm{\vv_1}_{\LL^6(\Omega)}^{4} + \norm{\vv_2}_{\LL^6(\Omega)}^{4} \Big) 
            \Big(\norm{\vv}_\sigma^{2}+\norm{c} _{\ast}^{2}\Big) .
    \end{split}%
\end{equation*}%
Applying Gronwall's lemma, and recalling that $\vv_i\in L^4(0,T;\LL^6(\Omega))$, $i=1,2$, we eventually obtain
\begin{align*}
    \norm{\vv(t)}_\sigma^{2}+\norm{c(t)} _{\ast}^{2}
    \le \big(\norm{\vv_0}_\sigma^{2}+\norm{c_0}_{\ast}^{2} \big)
        \exp\left(\int_0^t 
        \Big( 1 + \norm{\vv_1(s)}_{\LL^6(\Omega)}^{4} 
            + \norm{\vv_2(s)}_{\LL^6(\Omega)}^{4} \Big) \ds \right)
\end{align*}
for all $t\in [0,T]$. As the right-hand side vanishes if $\vv_{0,1}=\vv_{0,2}$ and $c_{0,1}=c_{0,2}$ a.e.~in $\Omega$, this proves the uniqueness of the corresponding strong solution.

\subsection{Uniform estimates for weak solutions} 
\label{SUB:EST:WEAK}

We now want to verify item \ref{E3} of Theorem~\ref{THM:WP:NONLOCAL}.
To this end, let $\eps>0$ be arbitrary, and let $(\vveps,\ceps,\mueps)$ be a corresponding weak solution to \eqref{ModelH} that can be constructed by a semi-Galerkin scheme explained in Subsection~\ref{SUB:EX}. We point out that all the following computations can be carried out rigorously within this semi-Galerkin scheme as the associated approximate solutions are sufficiently regular. Eventually, by passing to the limit in the approximation parameter, the obtained uniform bounds hold true for the considered weak solution $(\vveps,\ceps,\mueps)$.

From now on, the letter $C$ denotes generic positive constants that may depend only on the choice of $\Omega$, the initial data and the system parameters, but not on $\eps$.
The exact value of $C$ may vary throughout this proof.

Testing the equations \eqref{ModelH:NS} by $\vveps$, \eqref{ModelH:CH1} by $\mueps$ and \eqref{ModelH:CH2} by $\delt\ceps$, and using integration by parts, we derive the energy inequality
\begin{align}
E_\eps\big(\vveps(t),\ceps(t)\big)+\int_0^t \norm{D\vveps(s)}^2\ds +\int_0^t \norm{\Grad\mueps(t)}^2 \ds \leq E_\eps(\vv_{\eps,0},c_{\eps,0}) \le C_0
  \label{basic}
\end{align}
for all $t\ge 0$. 
We point out that, in a rigorous semi-Galerkin scheme, we initially merely obtain the local-in-time existence of an approximate solution. However, as this approximate solution fulfills a discrete version of \eqref{basic} as long as it exists, we can use this estimate to conclude that the approximate solution can actually be extended onto $[0,\infty)$.

Let now $T>0$ be arbitrary and let $C(T)$ denote generic positive constants that may depend only on $\Omega$, the initial data and the system parameters, but not on $\eps$. The exact value of $C(T)$ may vary in the subsequent line of argument.

In view of \ref{ASS:S1}, the boundedness of the energy resulting from \eqref{basic} already entails
\begin{align}
    \label{EST:CEPS:1}
    |\ceps| < 1 \quad\text{a.e.~in $\Omega_T$.}
\end{align}
Using this bound as well as Korn's inequality, we further conclude from \eqref{basic} that
\begin{align}
    \label{basic2}
    \begin{aligned}
    &\|\ceps\|_{L^\infty(0,T;L^\infty(\Omega))} 
    + \|\Grad\mueps\|_{L^2(0,T;\LL^2(\Omega))}
    \\
    &\quad
    + \|\vveps\|_{L^\infty(0,T;\LL^2(\Omega))}
    + \|\vveps\|_{L^2(0,T;\HH^1(\Omega))}
    \le C(T).
    \end{aligned}
\end{align}
We now recall the inequality
\begin{align}
\label{EST:F:1}
\intO |F^\prime(\ceps)|\dx \leq C\intO F^\prime(\ceps)(\ceps - \overline{c}_\eps)\dx + C
\quad\text{a.e.~in $[0,T]$,}
\end{align} 
which can be found, e.g., in \cite[Proposition~4.3.]{Miranville-book}.
Testing \eqref{ModelH:CH2} by $\ceps - \overline{c}_\eps$, we obtain
\begin{align}\label{sing7}
\intO \mueps(\ceps - \overline{c}_\eps)\dx = \intO \mathcal{L}_\eps \ceps (\ceps - \overline{c}_\eps)\dx + \intO \big( F^\prime(\ceps) - \theta_0\big)(\ceps - \overline{c}_\eps)\dx .
\end{align}
By the definition of the mean, the left-hand side can be reformulated as
\begin{align*}
    \intO \mueps(\ceps - \overline{c}_\eps)\dx = \intO (\mueps - \overline{\mu}_\eps)\ceps \dx.
\end{align*}
In view of the properties of $\mathcal{L}_\eps$, the first term on the right-hand side of \eqref{sing7} is nonnegative.
Due to \eqref{basic2}, the Poincar\'e--Wirtinger inequality yields
\begin{align}
\label{EST:F:2}
\left|\intO F^\prime(\ceps)(\ceps-\overline{c}_\eps)\dx\right| 
\leq C\big(1 + \|\Grad\mueps\|\big).
\end{align}
Hence, by means of \eqref{EST:F:1}, we conclude
\begin{align}
\label{EST:F:2*}
\intO |F^\prime(\ceps)|\dx 
\leq C\big(1 +  \|\Grad\mueps\|\big).
\end{align}
Consequently, it holds
\begin{align}
|\overline{\mu}_\eps| 
= \big|\overline{F^\prime(\ceps)} - \theta_0\overline{c}_\eps\big|
\leq C\big(1 + \|\Grad\mueps\|\big).
\label{overcontrol}
\end{align}
Recalling \eqref{basic2} and applying Poincar\'e's inequality, we thus conclude
\begin{align}
\label{basic3}
\|\mueps\|_{L^2(0,T;H^1(\Om))}
\le C(T).
\end{align}
By comparison in \eqref{ModelH:CH1}, we further have
\begin{align}
\norm{\delt \ceps}_{L^2(0,T;H^1(\Om)')}\leq C(T)
\label{dtc}
\end{align}
with the help of \eqref{basic2} and \eqref{basic3}.
Furthermore, using again \eqref{basic2} and \eqref{basic3}, and recalling the definition of $\mathbf{P}_\sigma$ (see~\ref{PRE:STOKES}), we deduce
\begin{equation*}
    \bignorm{\mathbf{P}_\sigma\big[\mueps\Grad\ceps\big]}_{\sigma}\leq \Vert \Grad\mueps\Vert . 
\end{equation*}
Performing the usual estimates for the remaining terms in the Navier--Stokes equation \eqref{ModelH:NS}, we conclude
\begin{equation*}
    \big\|\mathbf{P}_\sigma\big[(\vveps\cdot\Grad)\vveps +2\Div(\nu D\vveps)\big]\big\|_{L^{\frac4n}(0,T;\Vsigma')} \le C(T),
\end{equation*}
which directly yields
\begin{align}
\norm{\delt \vveps}_{L^\frac4n(0,T;\Vsigma')}\leq C(T)
 \label{dv}
\end{align}
by a further comparison argument. Combining \eqref{basic2}, \eqref{basic3}, \eqref{dtc} and \eqref{dv}, we have thus verified \eqref{regg-weak-est1}.
 
If $\Omega=\TTn$, we further test \eqref{ModelH:CH2} by $-\Lap\ceps$. After integrating by parts, we use the identity
\begin{align}
    \label{ID:LEPS}
    \begin{aligned}
    &\intO \Grad \big(\mathcal{L}_\eps\ceps\big) \cdot \Grad\ceps \dx 
    =\intO \mathcal{L}_\eps\Grad\ceps\cdot \Grad\ceps \dx 
    \\
    &\quad =\frac12\intO \intO J_\eps(x-y)\vert \Grad \ceps(x)-\Grad \ceps(y)\vert^2 \dx \dy
    \end{aligned}
\end{align}
to deduce
\begin{align}
    \begin{aligned}
    &\intO F''(\ceps)\vert \Grad \ceps\vert^2 \dx
    + \frac12\intO \intO J_\eps(x-y)\vert \Grad \ceps(x)-\Grad \ceps(y)\vert^2 \dx \dy
    \\
    &\qquad 
    = \intO \Grad\mueps\cdot\Grad\ceps \dx
    + \theta_0\norm{\Grad\ceps}^2.
    \end{aligned}
    \label{muu}
\end{align}
We point out that \eqref{ID:LEPS} follows from the relation $\Grad (J_\eps\ast \ceps) = J_\eps\ast \Grad\ceps$, which holds if $\Omega=\TTn$, but is (in general) not valid if $\Omega$ is a bounded domain.
Exploiting Lemma~\ref{LEM:EN}\ref{EN:IEQ:1} with $\gamma=\frac1{4\theta_0}$, we find $\eps_w=\eps_w(\theta_0)>0$ such that 
\begin{align}
\label{est:theta0}
 \theta_0\norm{\Grad\ceps}^2\leq \frac 1 4\intO \intO J_\eps(x-y)\vert \Grad \ceps(x)-\Grad \ceps(y)\vert^2 \dx \dy  +C\norm{\ceps}^2
\end{align}
if $\eps\in (0,\eps_w]$.
Combining \eqref{muu} and \eqref{est:theta0}, we infer
\begin{align}
    \begin{aligned}
    &\intO F''(\ceps)\vert \Grad \ceps\vert^2 \dx
    + \frac14\intO \intO J_\eps(x-y)\vert \Grad \ceps(x)-\Grad \ceps(y)\vert^2 \dx \dy
    \\
    &\qquad 
    = \intO \Grad\mueps\cdot\Grad\ceps \dx
    +C\norm{\ceps}^2. 
    \label{muu2}
    \end{aligned}
\end{align}
provided that $\eps\in (0,\eps_w]$.
Recalling \eqref{basic2} and that $F''\ge \theta$, we use Young's inequality to infer from \eqref{muu2} that
\begin{align}
    \begin{aligned}
    \theta\norm{\Grad \ceps}^2
    &\leq \intO F''(\ceps)\vert \Grad \ceps\vert^2 \dx 
    +\frac14\intO \intO J_\eps(x-y)\vert \Grad \ceps(x)-\Grad \ceps(y)\vert^2 \dx \dy 
    \\[1ex]
    &\leq C(1+\Vert \Grad\mueps\Vert^2) +\frac\theta2\norm{\Grad \ceps}^2,
    \end{aligned}
    \label{essen}
\end{align}
if $\eps\in (0,\eps_w]$.
In combination with \eqref{basic2}, we thus conclude
 \begin{align}
    \label{EST:CEPS:T}
     \norm{\ceps}_{L^2(0,T;H^1(\Om))}\leq C \quad\text{if $\Omega=\TTn$ and $\eps\in (0,\eps_w]$.}
 \end{align}
 This means that the uniform estimate \eqref{regg-weak-est2} is established and thus, property \ref{E3} is verified.
 
\subsection{Uniform estimates for strong solutions}\label{SUB:EST:STRONG}

Next, we intend to verify item \ref{K3} of Theorem~\ref{THM:WP:NONLOCAL}.
Therefore, we fix an arbitrary $\eps\in(0,\eps_w]$, and we consider the corresponding right-maximal strong solution solution $(\vveps,\peps,\ceps,\mueps)$ to \eqref{ModelH}.
Again, all the following computations can be carried out rigorously within the semi-Galerkin scheme mentioned in Subsection~\ref{SUB:EX} as the associated approximate solutions are sufficiently regular. Eventually, by passing to the limit in the approximation parameter, the obtained uniform bounds hold true for the considered strong solution $(\vveps,\peps,\ceps,\mueps)$.

\paragraph{Step 1: Uniform estimates for $\vveps$ and $\peps$.}
Our first step is to derive higher order bounds on the velocity field $\vveps$, which are uniform with respect to $\eps$. In the case $\Omega=\TTn$, it is well known that testing the momentum equation \eqref{ModelH:NS} by $-\Delta\vveps$ (instead of $A_S\vveps$) is sufficient to bound $\vveps$ in the $\HH^2(\Om)$-norm (see, e.g., \cite{Temam}). However, as we also want to cover the case of $\Omega$ being a bounded domain, we use a more general approach and test the momentum equation by $A_S\vveps=-\mathbf{P}_\sigma\Delta\vveps$. 
Performing this testing procedure and employing the well-known identity
\begin{equation*}
	\frac{1}{2}\|D\vveps(t)\|^2 = \frac{1}{2}\|D\vv_{\eps,0}\|^2 + \int_0^t \intO \delt \vveps\cdot A_S\vveps\dx\ds 
 \end{equation*}
for almost all $t\in [0,T_{\eps,\ast})$, we obtain
\begin{align}
    \begin{aligned}
    \frac{1}{2}\|D\vveps(t)\|^2 
    &=
    \frac{1}{2}\|D\vv_{\eps,0}\|^2
    + \int_0^t\intO \mueps\Grad \ceps\cdot A_S\vveps\dx\ds
    \\
    &\qquad - \int_0^t\intO (\vveps\cdot\Grad)\vveps\cdot A_S\vveps\dx \ds 
    - \int_0^t\intO  \vert D\vveps\vert^2 \dx\ds  .
    \end{aligned}
    \label{Asv}
\end{align}
for almost all $t\in [0,T_{\eps,\ast})$.

Performing an integration by parts, and applying H\"older's inequality, Young's inequality and \eqref{basic2}, we get
\begin{align}
\label{e2}
\Big|\intO \mueps\Grad \ceps\cdot A_S\vveps\dx\Big| &= 
	\Big|\intO \Grad\mueps \ceps\cdot A_S\vveps\dx\Big| 
	\leq C\|\Grad\mueps\|^2 +  \frac{1}{4}\|A_S\vveps\|^2.
\end{align}
If $n=2$, we use H\"older's inequality, Young's inequality, the Gagliardo--Nirenberg inequality as well as \eqref{basic2} to deduce
\begin{align}
\begin{aligned}
 &\Big|\intO (\vveps\cdot\Grad)\vveps\cdot A_S\vveps\dx\Big| \leq C\|\vveps\|_{\LL^4(\Om)}\|D\vveps\|_{\LL^4(\Om)}\|A_S\vveps\| 
 \\
 &\quad \leq C\norm{\vveps}^\frac12\norm{D\vveps}\norm{A_S\vveps}^\frac32
 \leq \frac{1}{4}\|A_S\vveps\| + C\|D\vveps\|^4.
\end{aligned}
 \label{e1:2}
\end{align}
In fact, if $\Omega=\TTn$, the above integral even vanishes (see, e.g., \cite[Lemma 3.1]{Temam}). 
In the case $n=3$, we proceed similarly to derive the estimate
\begin{align}
\begin{aligned}
&\Big|\intO (\vveps\cdot\Grad)\vveps\cdot A_S\vveps\dx\Big| \leq C\|\vveps\|_{\LL^6(\Om)}\|D\vveps\|_{\LL^3(\Omega)}\|A_S\vveps\| 
\\
&\quad \leq C\norm{D\vveps}^\frac32\norm{A_S\vveps}^\frac32
 \leq \frac{1}{4}\|A_S\vveps\|^2 + C\|D\vveps\|^6. \label{e1:3}
 \end{aligned}
\end{align}

Furthermore, testing the momentum equation by $\delt \vveps$, we derive the identity
\begin{align}
    \begin{aligned}
	\int_0^t \norm{\delt \vveps}^2 \ds
    &= - \int_0^t\intO (\vveps\cdot\Grad)\vveps\cdot \delt \vveps\dx\ds 
        + 2 \int_0^t\intO  \Lap \vveps \delt \vveps \dx\ds
    \\
    &\qquad 
        + \int_0^t\intO \mueps\Grad \ceps\cdot \delt \vveps\dx\ds.
    \end{aligned}
 \label{dtv}
\end{align}
for almost all $t\in [0,T_{\eps,\ast})$.
Using Young's inequality as well as integration by parts, we obtain
\begin{align}
    \label{e3}
    \intO \Delta\vveps\cdot \delt \vveps \dx 
    &\leq C_0\norm{A_S\vveps}^2+\frac{1}{4}\norm{\delt \vveps}^2,
    \\
    \label{e5}
    \intO \mueps\Grad \ceps\cdot \delt \vveps\dx
    = - \intO \ceps \Grad \mueps \cdot \delt \vveps\dx
    &\leq C\norm{\Grad\mueps}^2+\frac{1}{4}\norm{\delt \vveps}^2.
\end{align}
for some positive constant $C_0$ depending on the same quantities as the constants denoted by $C$. 
Now, proceeding similarly as in the derivation of \eqref{e1:2} and \eqref{e1:3}, we deduce
\begin{align}
    \label{e4}
   \left \vert \intO (\vveps\cdot\Grad)\vveps\cdot \delt \vveps\dx \right\vert 
   \leq \frac{1}{4}\|\delt \vveps\|^2 +C_0 \norm{A_S\vveps}^2+C\|D\vveps\|^\gamma,
\end{align}
where $\gamma=4$ if $n=2$ and $\gamma=6$ if $n=3$.
We now add inequality \eqref{Asv} and inequality \eqref{dtv} multiplied by $\frac{1}{8C_0}$. By means of \eqref{basic3},  \eqref{e1:2}--\eqref{e2} and \eqref{e3}--\eqref{e5}, we infer
\begin{align*}
	&\frac{1}{2}\|D\vveps(t)\|^2 
        +\frac{1}{16C_0} \int_0^t\norm{\delt \vveps}^2 \ds
        +\frac{1}{16} \int_0^t\|A_S\vveps\|^2 \ds
    \\
    &\quad \leq C + \frac{1}{2}\|D\vv_{\eps,0}\|^2
        + C \int_0^t \|D\vveps\|^\gamma \ds,
\end{align*}
for almost all $t\in [0,T_{\eps,\ast})$ with $\gamma$ as introduced above.
We now recall \eqref{basic2}, the uniform estimates stated in \ref{E3} as well as assumption \eqref{initiald:1}.
Applying Bihari's inequality (see, e.g., \cite[Lemma~II.4.12]{Boyer}), we conclude that the estimate
\begin{align}
\norm{\vveps}_{L^\infty(0,T;\Vsigma)}+\norm{\vveps}_{L^2(0,T;\HH^2(\Om))}+\norm{\delt \vveps}_{L^2(0,T;\LL^2_\sigma(\Om))}\leq C(T)
\label{d1}
\end{align}
holds for all $T\in (0,T_{\eps,\ast})$. In the case $n=3$, due to the uniform bound assumed in 
\eqref{initiald:1}, Bihari's inequality (as stated in \cite[Lemma~II.4.12]{Boyer}) implies the existence of a time $T_\ast\in(0,T_{\eps,\ast})$, which is independent of $\eps$, such that \eqref{d1} holds true for all $T\in (0,T_\ast]$. By a comparison argument, we eventually obtain
\begin{align}
    \label{d2}
    \norm{\peps}_{L^2(0,T;H^1(\Omega))} \le C(T)
\end{align}
for any $T\in (0,\infty)$ if $n=2$ (since then $T_{\eps,*}=\infty$) and any $T\in (0,T_*]$ if $n=3$.

\paragraph{Step 2: Uniform estimates for $\ceps$ and $\mueps$.}
In the following, let $T\in (0,\infty)$ if $n=2$ and $T\in (0,T_*]$ if $n=3$ be arbitrary.
The next goal is to derive further uniform bounds on $\ceps$ and $\mueps$. 
Therefore, in order to obtain the desired estimates, we need to truncate the initial datum $c_{\eps,0}$ as it was done in \cite{PS}. For any $k\in\N$, we define the Lipschitz continuous truncation
\begin{equation*}
    \sigma_k: \R\to\R, \quad s\mapsto
    \begin{cases}
        -1 + \tfrac 1k &\text{if $s< -1 + \tfrac 1k$}, \\
        s &\text{if $ -1 + \tfrac 1k < s < 1 - \tfrac 1k$}, \\
        1 - \tfrac 1k &\text{if $ s > 1 - \tfrac 1k$},
    \end{cases}
\end{equation*}
and we set $c_{\eps,0}^k := \sigma_k \circ c_{\eps,0}$ and $\mu_{\eps,0}^k := \mathcal{L}_\eps c_{\eps,0}^k + f'(c_{\eps,0}^k)$. 
The strong solution corresponding to initial datum $(\vv_{\eps,0},c_{\eps,0}^k)$ will be denoted as $(\vvepsk,\pepsk,\cepsk,\muepsk)$. 
We point out that the estimates \eqref{basic}, \eqref{basic2}, \eqref{basic3}, \eqref{dtc}, \eqref{dv}, \eqref{d1} and \eqref{d2} remain valid for the solution $(\vvepsk,\pepsk,\cepsk,\muepsk)$ and are uniform with respect to $k$ as long as $k$ is sufficiently large. 
Indeed, for the $\eps\in(0,\eps_s]$ chosen above, there exists $k_0=k_0(\eps)$ such that we have ${E}_\eps(\vv_{\eps,0},c_{\eps,0}^k)\leq {E}_\eps(\vv_{\eps,0},c_{\eps,0})+C$ for all $k\ge k_0$, and clearly the initial datum $\vv_{\eps,0}$ does not depend on $k$. In the following, we thus consider $k\geq k_0(\eps)$.

Due to the assumptions in Theorem~\ref{THM:WP:NONLOCAL}, we clearly have $\mu_{\eps,0}^k \in H^1(\Omega)$.
Moreover, as shown in \cite[Formula~(3.9)]{PS}, the function $\muepsk$ has the additional regularity
\begin{equation}
    \label{REG:MUEPSK}
    \muepsk \in C([0,T];H^1(\Omega)).
\end{equation}
In view of the regularities stated in \eqref{regg}, we further have
\begin{equation}
    \label{REG:CEPSK}
    \cepsk \in C_\mathrm{w}([0,T];H^1(\Omega))
\end{equation}
thanks to an embedding result, which can be found, e.g., in \cite[Corollary~2.1]{Strauss}.

Arguing as in \cite[Proof of Theorem~4.1]{AGGP} ($n=2$) or \cite[Proof of Theorem~2.2]{PS} ($n=3$), we derive the identity
\begin{align}
    \label{EQ:DTGMU}
    \begin{aligned}
	&\frac{1}{2}\|\Grad\muepsk(t)\|^2 
        + \int_0^t\intO \vvepsk\cdot\Grad \cepsk \, \delt \muepsk\dx\ds
    \\
    &\qquad 
        + \int_0^t\intO \mathcal{L}_\eps\delt \cepsk \, \delt \cepsk\dx\ds
        + \int_0^t\intO F^{\prime\prime}(\cepsk)|\delt \cepsk|^2\dx\ds 
    \\
    &\quad
    = \frac{1}{2}\|\Grad\mu_{\eps,0}^k\|^2 + \int_0^t\intO \theta_0\vert \delt \cepsk \vert^2\dx\ds 
    \end{aligned}
\end{align}
for almost all $t\in [0,T]$.
Formally, \eqref{EQ:DTGMU} can be obtained as follows:
We differentiate \eqref{ModelH:CH2} with respect to time, which yields
\begin{align}
    \label{EQ:DTMU}
    \delt \muepsk = \mathcal{L}_\eps \delt\cepsk + f''(\cepsk)\delt\cepsk
\end{align}
Now, we test \eqref{ModelH:CH1} by $-\delt\muepsk$ and \eqref{EQ:DTMU} by $\delt\cepsk$. Adding and integrating the resulting equations with respect to time from $0$ to $t$, and using the identity
\begin{align*}
    \frac{1}{2}\|\Grad\muepsk(t)\|^2 
    &= \frac{1}{2}\|\Grad\mu_{\eps,0}^k\|^2 - \int_0^t\intO \delt\muepsk\, \Lap\muepsk \dx\ds
    \\
    &= \frac{1}{2}\|\Grad\mu_{\eps,0}^k\|^2 - \int_0^t\intO \delt\cepsk\, \delt\muepsk \dx\ds,
\end{align*}
we arrive at \eqref{EQ:DTGMU}.

By a straightforward computation, the second summand on the left-hand side of \eqref{EQ:DTGMU} can be reformulated as
\begin{align}
    \label{EQ:VCM}
    \begin{aligned}
	&\int_0^t\intO \vvepsk\cdot\Grad \cepsk\delt \muepsk\dx\ds 
    \\
    &\quad 
    = \intO \vvepsk(t)\cdot\Grad \cepsk(t)\muepsk(t)\dx 
        - \intO \vv_{\eps,0}^k \cdot\Grad c_{\eps,0}^k \mu_{\eps,0}^k \dx
    \\
    &\qquad 
        - \int_0^t \intO \delt \vvepsk\cdot\Grad \cepsk\muepsk\dx\ds
        + \int_0^t \intO \vvepsk\delt \cepsk\cdot\Grad\muepsk\dx\ds
    \end{aligned}
\end{align}
for almost all $t\in [0,T]$.
Due to the properties of $\mathcal{L}_\eps$, we further have
\begin{equation}
    \label{EST:LDTC}
    \intO \mathcal{L}_\eps\delt \cepsk\; \delt \cepsk\dx
    = 2 \mathcal{E}_\eps(\delt \cepsk)
\end{equation}
a.e.~in $[0,T]$.
We now introduce the function
\begin{align*}
	H_\eps^k:\R\to\R,\quad t\mapsto \frac{1}{2}\|\Grad\muepsk(t)\|^2 - \intO \vvepsk(t) \cdot\Grad \muepsk(t)\, \cepsk(t) \dx.
\end{align*} 
Due to the regularities \eqref{regg}, \eqref{REG:MUEPSK} and \eqref{REG:CEPSK}, we know  that $H_\eps^k$ is continuous and it thus holds
\begin{equation}
    \label{EQ:HE:0}
    \begin{aligned}
	H_\eps^k(0) 
    &= \frac{1}{2}\|\Grad\mu_{\eps,0}^k\|^2 - \intO \vv_{\eps,0} \cdot\Grad \mu_{\eps,0}^k \,c_{\eps,0}^k \dx.
    \end{aligned}
\end{equation} 
Recalling $F''\ge \theta$, combining \eqref{EQ:DTGMU}, \eqref{EQ:VCM} and \eqref{EST:LDTC}, and using integration by parts, we conclude
\begin{align}\label{sing1}
    \begin{aligned}
	&H_\eps^k(t)  
        + \int_0^t \intO \theta |\delt\cepsk|^2 \dx\ds
        + 2\int_0^t \mathcal{E}_\eps(\delt \cepsk)\ds
    \\[1ex]
    &\quad \le H_\eps^k(0)
        + \int_0^t \intO \delt \vvepsk\cdot\Grad \cepsk\muepsk \dx\ds
    \\
    &\qquad
        - \int_0^t \intO \vvepsk\delt \cepsk\cdot\Grad\muepsk\dx\ds
        + \int_0^t \intO \theta_0\vert \delt \cepsk\vert^2\dx\ds .
    \end{aligned}
\end{align}
for almost all $t\in[0,T]$.
Exploiting \eqref{basic2} and using integration by parts as well as Young's inequality, we obtain
\begin{align}
    \label{EST:HE:1}
    &\left|\intO \delt \vvepsk\cdot\Grad \cepsk\muepsk\dx\right| 
    = \left|\intO \delt \vvepsk\cdot \cepsk\Grad\muepsk\dx\right|
    \leq C\norm{\Grad\muepsk}^2+C\norm{\delt \vvepsk}^2.
\end{align}
Moreover, invoking the continuous embedding $\HH^2(\Om)\hookrightarrow\LL^\infty(\Om)$ and Young's inequality, we deduce
\begin{align}
    \label{EST:HE:2}
    \begin{aligned}
    \left|\intO \vvepsk\cdot\delt \cepsk\Grad\muepsk\dx\right| 
    &\leq C\norm{\vvepsk}_{\LL^\infty(\Om)}\norm{\delt \cepsk}\norm{\Grad\muepsk}
    \\
    &\leq \theta_0\|\delt \cepsk\|^2 + C\|\vvepsk\|_{\HH^2(\Om)}^2\|\Grad\muepsk\|^2.
    \end{aligned}
\end{align}
Applying Lemma~\ref{LEM:EN}\ref{EN:IEQ:2} with $\gamma=\frac{1}{2\theta_0}$, we find $\eps_s=\eps_s(\theta_0) > 0$ such that
\begin{align}
    2\theta_0\norm{\delt \cepsk}^2
    \leq  \mathcal{E}_\eps(\delt \cepsk) 
        + C\norm{\delt \cepsk}_{*}^2
\label{poinn}
\end{align}
if $\eps\in (0,\eps_s]$. Without loss of generality, we assume $\eps_s\le \eps_w$ and from now on, we further demand that $\eps \in (0,\eps_s]$.

Combining the inequalities \eqref{sing1}, \eqref{EST:HE:1}, \eqref{EST:HE:2} and \eqref{poinn}, we conclude that the estimate
\begin{equation*}
	\begin{aligned}
	&H_\eps^k(t)  
        + \theta \int_0^t \norm{\delt\cepsk}^2 \ds
        + \int_0^t \mathcal{E}_\eps(\delt \cepsk) \ds
    \\[1ex]
    &\quad \le H_\eps^k(0)
        + C \int_0^t \|\delt \vvepsk\|^2 \ds
        + C \int_0^t \norm{\delt \cepsk}_{*}^2 \ds
    \\
    &\qquad
        + C \int_0^t \Big(1 + \|\vvepsk\|_{\HH^2(\Om)}^2\Big) \|\Grad\muepsk\|^2 \ds.
    \end{aligned}
\end{equation*}
holds for almost all $t\in[0,T]$.
Note that, due to \eqref{basic2}, we have 
\begin{align}\label{sing2}
\left|\intO \vvepsk\cdot \cepsk\Grad\muepsk\dx\right| \leq \norm{\vvepsk}\norm{\Grad\muepsk}\leq C + \frac14\|\Grad\muepsk\|^2.
\end{align}
Hence, there exist positive constants $K$ and $\tilde K$ that may depend on the same quantities as $C$ such that
\begin{align*}
	\frac{1}{4}\|\Grad\muepsk(t)\|^2 - \tilde{K}
    \leq H_\eps^k(t)
    \leq  \|\Grad\muepsk(t)\|^2 + K 
    \quad\text{for almost all $t\in [0,T]$}.
\end{align*}
This allows us to apply Gronwall's lemma, which yields
\begin{align}
    \begin{aligned}
    &H_\eps^k(t) + \theta \int_{0}^{t}\|\delt  \cepsk(s)\|_{L^2(\Om)}^2 \ds
	\\
    &\leq \left(H_\eps^k(0) + C\int_{0}^t(\|\delt \vvepsk(s)\|+\norm{\delt \cepsk}_{*}^2)\ds\right)\exp\left(\int_{0}^{t}C\big(1 + \|\vvepsk(s)\|_{\HH^2(\Om)}^2\big)\ds \right) 
    \\[1ex]
    &\leq C(T)\big(1+H_\eps^k(0)\big)
    \end{aligned}
 \label{gronwall1}
\end{align}
for almost all $t\in [0,T]$, thanks to \eqref{basic2}, \eqref{dtc} and \eqref{d1}.
It thus remains to control $H_\eps^k(0)$ uniformly with respect to $\eps$. 
Recalling the representation \eqref{EQ:HE:0} as well as the assumptions on  $c_{\eps,0}^k$, we deduce
\begin{align}
\label{convergencemu}
	|H_\eps^k(0)| 
    \le \norm{\Grad\mu_{\eps,0}^k}^2 
        + \norm{\vv_{\eps,0}} \norm{\Grad\mu_{\eps,0}^k} \norm{c_{\eps,0}^k}_{L^\infty(\Omega)}\leq \norm{\Grad\mu_{\eps,0}^k}^2 
        + \norm{\vv_{\eps,0}} \norm{\Grad\mu_{\eps,0}^k},
\end{align}
recalling $\vert c_{\eps,0}^k\vert<1$. Now, following \cite{PS,AGGP}, we can prove that $\norm{\nabla\mu_{\eps,0}^k}\to \norm{\nabla{\mu}_{\eps,0}}$ as $k\to \infty$. Therefore, for the $\eps\in(0,\eps_s]$ that was chosen above, there exists $\overline{k}=\overline{k}(\eps)\geq k_0(\eps)$ such that 
\begin{equation*}
    \norm{\nabla{\mu}_{\eps,0}^k-\nabla{\mu}_{\eps,0}}\leq \frac{1}{2} \quad \text{for all $k\geq \overline{k}(\eps)$}.
\end{equation*}
Hence, from \eqref{convergencemu} and the assumptions on $\mu_{\eps,0}$ and $\vv_{\eps,0}$, we conclude
\begin{align*}
   \vert H_\eps^k(0)\vert
   &\leq 2{\norm{\nabla{\mu}_{\eps,0}^k-\nabla{\mu}_{\eps,0}}}^2+2\norm{\nabla{\mu}_{\eps,0}}^2
   \\
   &\quad +C\norm{D\vv_{\eps,0}}({\norm{\nabla{\mu}_{\eps,0}^k-\nabla{\mu}_{\eps,0}}}+\norm{\nabla{\mu}_{\eps,0}})
   \\
   &\leq C,
\end{align*}
for any $k\geq \overline{k}(\eps)$. 
Consequently, for every $k\geq \overline{k}(\eps)$, \eqref{gronwall1} provides the bound
\begin{equation}
    \label{EST:GMUEPS:K}
    \norm{\Grad\muepsk}_{L^\infty(0,T;\LL^2(\Omega))} + \norm{\delt\cepsk}_{L^2(0,T;L^2(\Omega))} \le C(T).
\end{equation}
As the estimates \eqref{basic2}, \eqref{basic3}, \eqref{dtc}, \eqref{dv}, \eqref{d1} and \eqref{d2} remain valid for the solution $(\vvepsk,\pepsk,\cepsk,\muepsk)$ and are uniform with respect to $k$, it follows by standard compactness arguments that $(\vvepsk,\pepsk,\cepsk,\muepsk)$ converges to $(\vveps,\peps,\ceps,\mueps)$, as $k\to\infty$, in the corresponding function spaces. For more details, we also refer to \cite[Proof of Theorem~2.2]{PS}.
In particular, we conclude that the strong solution $(\vveps,\peps,\ceps,\mueps)$ satisfies the uniform bound
\begin{equation}
    \label{EST:GMUEPS}
    \norm{\Grad\mueps}_{L^\infty(0,T;\LL^2(\Omega))} + \norm{\delt\ceps}_{L^2(0,T;L^2(\Omega))} \le C(T).
\end{equation}
Using Poincar\'e's inequality along with \eqref{overcontrol}, we further obtain 
\begin{equation}
    \label{EST:MUEPS}
    \norm{\mueps}_{L^\infty(0,T;H^1(\Omega))} \le C(T).
\end{equation}

In the case $\Omega=\TTn$, it further follows from \eqref{essen} that
\begin{equation}
    \label{EST:CEPS:T2}
    \norm{\ceps}_{L^{\infty}(0,T;H^1(\Omega))} 
    \le C(T).
\end{equation}
By comparison in \eqref{ModelH:CH1}, we now use the uniform estimates \eqref{d1} and \eqref{EST:CEPS:T2} to deduce
\begin{align*}
    &\norm{\Lap\mueps}_{L^2(0,T;L^2(\Omega))} 
    \\
    &\quad\le \norm{\delt\ceps}_{L^2(0,T;L^2(\Omega))}
        + \norm{\vveps}_{L^2(0,T;\LL^\infty(\Omega))}\, \norm{\Grad\ceps}_{L^\infty(0,T;\LL^2(\Omega))}
    \le C(T).
\end{align*}
As $\Omega=\TTn$, we have $\norm{D^2 \mu_\eps}^2 = \norm{\Lap \mu_\eps}^2$ a.e.~in $[0,T]$.
Hence, in combination with \eqref{EST:MUEPS}, we conclude the uniform bound
\begin{equation}
    \label{L2dtc}
    \norm{\mueps}_{L^2(0,T;H^2(\Omega))}
    \le C(T).
\end{equation}

\paragraph{Step 3: A uniform estimate for $F'(\ceps)$.}

In the following, let $p\in[2,\infty)$ if $n=2$ and let $p\in [2,6]$ if $n=3$.
As a consequence of \eqref{EST:MUEPS}, we obtain the estimate
\begin{align}
    \norm{\mueps}_{L^\infty(0,T;L^p(\Om))}
    \leq C(T)\sqrt{p}.
    \label{muessential}
\end{align}
In the case $n=3$, this inequality simply follows from the continuous embedding $H^1(\Omega)\emb L^p(\Omega)$ and the fact that $\sqrt{p} \ge 1$.
In the case $n=2$, \eqref{L2dtc} follows from the following Sobolev type inequality, which can be found, e.g., in \cite[p.~479]{Trudinger}: there exists a constant $C_\Omega>0$ depending only on $\Omega$ such that for all $u\in H^1(\Omega)$ and all $p\in[2,\infty)$, it holds
\begin{align*}
	\|u\|_{L^p(\Omega)} \leq C_\Omega \sqrt{p}\, \|u\|_{H^1(\Omega)}.
\end{align*}

We now intend to derive a uniform bound on $F^\prime(\ceps)$ in $L^\infty(0,T;L^p(\Omega))$.
Therefore, we test equation \eqref{ModelH:CH2} by $|F^\prime(\ceps)|^{p-2}F^\prime(\ceps)$. If $p=2$, this test function is simply to be interpreted as $F^\prime(\ceps)$. We obtain
\begin{align*}
	\intO \mueps|F^\prime(\ceps)|^{p-2}F^\prime(\ceps)\dx &= \intO \mathcal{L}_\eps \ceps|F^\prime(\ceps)|^{p-2}F^\prime(\ceps)\dx + \|F^\prime(\ceps)\|_{L^p(\Om)}^p 
    \\
	&\quad-  \theta_0\intO \ceps|F^\prime(\ceps)|^{p-2}F^\prime(\ceps)\dx.
\end{align*}
Using H\"older's and Young's inequalities, and recalling that $\vert\ceps\vert<1$ a.e.~in $\Omega_T$, we observe
\begin{align*}
	\intO \mueps|F^\prime(\ceps)|^{p-2}F^\prime(\ceps)\dx 
    &\leq C\|\mueps\|_{L^p(\Om)}^p 
        + \frac14\|F^\prime(\ceps)\|_{L^p(\Om)}^p, 
    \\
	\theta_0\intO \ceps|F^\prime(\ceps)|^{p-2}F^\prime(\ceps)\dx 
    &\leq  C|\Omega|^\frac1p
        + \frac14\|F^\prime(\ceps)\|_{L^p(\Om)}^p\leq C+\frac14\|F^\prime(\ceps)\|_{L^p(\Om)}^p.
\end{align*}
Since $F'$ is strictly increasing, so is $g(r) := |F^\prime(r)|^{p-2}F^\prime(r)$ for $r\in (-1,1)$. This implies
\begin{equation*}
    \big(\ceps(x) - \ceps(y)\big)\big[g(\ceps(x)) - g(\ceps(y))\big] \ge 0
    \quad\text{for almost all $x,y\in\Omega$.}
\end{equation*}
Consequently, since $J_\eps \ge 0$, we have
\begin{align*}
	&\intO \mathcal{L}_\eps  \ceps|F^\prime(\ceps)|^{p-2}F^\prime(\ceps)\dx \\
	&= \frac12\intO \intO J_\eps(x-y)\, \big(\ceps(x) - \ceps(y)\big)\big[g(\ceps(x)) - g(\ceps(y))\big]\dy\dx 
    \ge 0.
\end{align*}
Altogether, this implies
\begin{align*}
	\|F^\prime(\ceps)\|_{L^p(\Om)} \leq C\big(1+ \|\mueps\|_{L^p(\Om)}\big).
\end{align*}
Hence, in combination with \eqref{muessential}, we eventually conclude
\begin{align}
\norm{F^\prime(\ceps)}_{L^\infty(0,T;L^p(\Om))}\leq C\sqrt{p}.
    \label{FLp}
\end{align}
Having all these uniform estimates at hand, item \ref{K3} is now verified.

\subsection{Strict separation property}
The last step is to verify the strict separation properties stated in item~\ref{K4}.

In the case $n=2$, it has already been proven in \cite[Theorem 1.4]{AGGP} that 
\begin{align}
\norm{F'(\ceps)}_{L^\infty(0,\infty;L^p(\Om))}\leq C_\eps \sqrt{p},\quad \text{for all } p\in[2,\infty),
\label{Lpinfty}
\end{align}
where $C_\eps$ is a constant that may depend on the usual quantities as well as on $\eps$.
Therefore, one can proceed as in the the proof of \cite[Theorem 4.3]{GP} to conclude that \eqref{delt} holds.
Assuming that the initial datum is strictly separated, we can repeat the same argument as in \cite[Corollary 4.5]{P} (i.e. the De Giorgi iteration scheme without the use of a cutoff function in time) to show that there exists $T_S>0$ such that the solution is strictly separated on $[0,T_S]$. Combined with \eqref{delt}, the result \eqref{SP-Star} is verified. We point out that the dependence of $\delta_\eps^*$ on $\eps$ is not only due to the constant $C_\eps$ in estimate \eqref{Lpinfty} (which could be avoided if we restrict ourselves to finite time intervals, see~\eqref{FLp}), but also results from the fact that the $W^{1,1}(X)$-norm of $J_\eps$ is not bounded uniformly with respect to $\eps$.

In the case $n=3$, thanks to estimate \eqref{FLp}, one can argue exactly as in \cite[Theorem 4.3]{P} (see also \cite[Remarks 4.7, 4.9]{P}) to prove the validity of \eqref{delt} and \eqref{SP-Star}. As in the two-dimensional setting, the dependence of $\delta_\eps^*$ on $\eps$ cannot be avoided using this method.

We remark that in the aforementioned proofs, the presence of the additional convective term $\vveps\cdot \Grad \ceps$ in \eqref{ModelH:CH1} does not disturb the line of argument, since in the De Giorgi iteration scheme this term simply vanishes as the velocity field is divergence-free and vanishes at the boundary if $\Omega$ is a bounded domain. For more details, we refer to \cite[Remark~4.7]{P}.

In summary, all statements of Theorem~\ref{THM:WP:NONLOCAL} are now established, and thus the proof is complete.
\hfill$\Box$


\section{Proof of Theorem \ref{THM:NLTL}}
\label{PROOF:NLTL}

Let $\eps_s>0$ be the real number introduced in Theorem~\ref{THM:WP:NONLOCAL}.
For any $\eps\in (0,\eps_s]$, let $(\vveps, \peps, \ceps, \mueps)$ be the unique right-maximal strong solution to the nonlocal Model~H (i.e., \eqref{ModelH} with $\eps\in (0,\eps_s]$) given by Theorem \ref{THM:WP:NONLOCAL}. Moreover, let 
$(\vv,p,c,\mu)$ denote the unique strong solution to the local Model~H (i.e., \eqref{ModelH} with $\eps=0$) given by Proposition \ref{PROP:WP:LOCAL}.

Note that the definition of $T_\diamond>0$ ensures that the strong solutions 
$(\vveps, \peps, \ceps, \mueps)$ with $\eps\in (0,\eps_s]$ and the strong solution $(\vv,p,c,\mu)$ exist on the time interval $[0,T_\diamond)$.
In particular, the strong solution $(\vv,p,c,\mu)$ fulfills the strict separation property stated in Proposition~\ref{PROP:WP:LOCAL}\ref{K3loc}, 
and for any $\eps\in (0,\eps_s]$, the strong solution $(\vveps, \peps, \ceps, \mueps)$ fulfills the uniform estimates stated in Theorem~\ref{THM:NLTL}\ref{E3} and \ref{K3}.

From now on, in order to verify the convergence property \eqref{finale}, let $T\in (0,T_\diamond)$ and $\eps\in (0,\eps_s]$ be arbitrary. Moreover, we use the notation
\begin{align*}
    (\tilde\vv_{0},\tilde c_{0}) 
    &:= (\vv_{0,\eps},c_{0,\eps}) - (\vv_{0},c_{0}),
    \\
    (\tvv,\tp,\tc,\tmu) 
    &:= (\vveps,\peps,\ceps,\mueps) - (\vv,p,c,\mu).
\end{align*}
This means that the quadruplet $(\tvv,\tp,\tc,\tmu)$ fulfills the following system of equations in the strong sense:
\begin{subequations}
\label{C:ModelH}
\begin{align}
    \label{C:ModelH:NS}
	&\delt \tvv + (\vveps \cdot \Grad)\vveps + (\vv \cdot \Grad)\vv - \Lap\tvv + \Grad \tp 
    = \mu_\eps \Grad c_\eps - \mu\Grad c,
    \quad \Div(\tvv) = 0
    &&\text{in}\;\Omega_T,
    \\
    \label{C:ModelH:CH1}
	&\delt \tc + \vveps\cdot\Grad \ceps + \vv\cdot\Grad c = \Delta\tmu
    \qquad&&\text{in}\;\Omega_T, 
    \\
    \label{C:ModelH:CH2}
	&\tmu = \mathcal{L}_\eps \ceps + \Lap c + F^{\prime }(c_{1})-F^{\prime }(c _{2})+\theta_0 \tc
    \qquad&&\text{in}\;\Omega_T,
    \\
    \label{C:ModelH:IC}
    &\tvv\vert_{t=0} = \tvv_0, \quad \tc\vert_{t=0} = \tc_0
    &&\text{in}\; \Omega.
\end{align}
If $\Omega$ is a bounded domain, $(\tvv,\tp,\tc,\tmu)$ also satisfies the boundary conditions
\begin{align}
    \label{C:ModelH:BC}
    \tvv = 0,\quad \deln \tmu =0
    \quad\text{on}\; \Gamma_T.
\end{align}
\end{subequations}

\paragraph{Step 1: An estimate for the difference $f^\prime(\ceps)-f^\prime(c)$.}
We first intend to derive an estimate for the difference $f^\prime(\ceps)-f^\prime(c)$ in the $L^1(\Omega)$-norm. Therefore, we exploit the strict separation property of the solution $(\vv,p,c,\mu)$.
Let $\delta_\star$ be the constant from \eqref{SP-Starloc} and let $\delta_0$ be the constant from \eqref{SP-Starloc2}. In the following, we set $\delta:=\delta_\star/2$ if $n=2$ and $\delta:=\delta_0/4$ if $n=3$. Hence, in view of Proposition~\ref{PROP:WP:LOCAL}\ref{K3loc}, we have,
\begin{align}
    \label{SEP:DELTA}
    \norm{c(t)}_{L^\infty(\Omega)} \le 1 - 2\delta\quad\text{for all $t\in [0,T]$}
\end{align}
due to the choices of $T_\diamond$ and $T$.
For any $t\in[0,T]$, we now define 
\begin{align*}
A_\delta(t)& := \big\{x\in\Omega:\ \vert \ceps(x,t)\vert\geq 1-\delta \big\},
\\
B_\delta(t)&:= \big\{x\in\Omega:\ \vert c(x,t)-\ceps(x,t)\vert \geq \delta \big\}.
\end{align*} 
Exploiting \eqref{SEP:DELTA}, we observe
\begin{align}
 \label{bounded}
   1-\delta
   \leq \vert \ceps(x,t)\vert
   \leq \vert c(x,t)\vert+\vert c(x,t)-\ceps(x,t)\vert 
   \leq 1-2\delta+\vert c(x,t)-\ceps(x,t)\vert
\end{align} 
for all $t\in [0,T]$ and all $x\in A_\delta(t)$.
This entails that
\begin{align}
 \vert c(x,t)-\ceps(x,t)\vert \geq \delta
 \label{bounded2}
\end{align}
for all $t\in [0,T]$ and all $x\in A_\delta(t)$.
Consequently, for every $t\in [0,T]$, we have the inclusion
\begin{equation*}
A_\delta(t)\subset B_\delta(t).
\end{equation*}
Therefore, invoking Chebyshev's inequality, we conclude 
\begin{align}
  &\vert A_\delta(t)\vert
  \leq \int_{B_\delta(t)} 1 \dx  
  \leq \int_{B_\delta(t)} \dfrac{\vert \ceps(t)-c(t)\vert^2}{\delta^2} \dx 
  \leq \intO \dfrac{\vert \ceps(t)-c(t)\vert^2}{\delta^2} \dx 
  \label{chebyshev}
\end{align}
for all $t\in [0,T]$, where $\vert A_\delta(t)\vert$ denotes the $n$-dimensional Lebesgue measure of the set $A_\delta(t)$.
Using the Cauchy--Schwarz inequality as well as the fundamental theorem of calculus, we deduce 
\begin{align}
    \begin{aligned}
    & \norm{f'(\ceps)-f'(c)}_{L^1(\Omega)}
    \\[1ex]
    &\leq \norm{f'(\ceps)-f'(c)}_{L^1(A_\delta)}+\norm{f'(\ceps)-f'(c)}_{L^1(\Omega\setminus A_\delta)}
    \\
    &\leq\norm{f'(\ceps)-f'(c)}_{L^2(A_\delta)}\vert A_\delta\vert^\frac12
        \, +\int_{\Omega\setminus A_\delta}\left\vert\int_0^1 f''\big(s\ceps+(1-s)c\big)(\ceps-c)\ds\right\vert \dx 
    \end{aligned}
    \label{contt}
\end{align}
for all $t\in [0,T]$. By \eqref{SEP:DELTA} and the definition of $A_\delta$, we have 
\begin{equation*}
    |s\ceps(t)+(1-s)c(t)| \le s |\ceps(t)| + (1-s) |c(t)| \le 1 - \delta 
    \quad\text{a.e.~in $\Omega\setminus A_\delta(t)$}
\end{equation*}
for all $t\in [0,T]$ and all $s\in [0,1]$.
Recalling $F''\in C(-1,1)$, we thus have
\begin{align}
    \begin{aligned}
      &\int_0^1 f''\big(s\ceps(t)+(1-s)c(t)\big)(\ceps(t)-c(t)) \ds
      \\
      &\quad \leq \left(\max_{\vert s\vert \leq 1-\delta}F''(s)+\theta_0\right)\vert\ceps(t)-c(t)\vert =: C_\delta  \vert\ceps(t)-c(t)\vert
      \quad\text{a.e.~in $\Omega\setminus A_\delta(t)$}.
    \end{aligned}
  \label{control_below}
\end{align}
Plugging this estimate into \eqref{contt} and using again \eqref{chebyshev}, we deduce 
\begin{align}
    \begin{aligned}
    &\norm{f'(\ceps)-f'(c)}_{L^1(\Omega)}
    \\
    &\leq\norm{f'(\ceps)-f'(c)}_{L^2(A_\delta)}\vert A_\delta\vert^\frac12+\int_{\Omega\setminus A_\delta}\left\vert\int_0^1 f''(s\ceps+(1-s)c)(\ceps-c)\ ds\right\vert\ \dx 
    \\
    &\leq 
    \frac{1}{\delta}\norm{f'(\ceps)-f'(c)}\norm{\ceps-c}+C_\delta\norm{\ceps-c}_{L^1(\Omega)}
    \\
    &\leq 
      \frac{1}{\delta}\big(\norm{f'(\ceps)}+\norm{f'(c)}\big)\norm{\ceps-c}+C_\delta\norm{\ceps-c}_{L^1(\Omega)}
    \\
    &\leq \left(\frac{C}{\delta} + \abs{\Omega}^{\frac{1}{2}} C_\delta\right)\norm{\ceps-c}=:K_\delta\norm{\ceps-c}.
    \end{aligned}
    \label{contt2}
\end{align}
Here, we used that $f'(c) \in L^\infty(0,T;L^2(\Omega))$ (see~\eqref{reggloc}) and that $f'(\ceps)$ is bounded in $L^\infty(0,T;L^2(\Omega))$ uniformly with respect to $\eps$ (see~\eqref{regg}).

\paragraph{Step~2: Estimates for the Navier--Stokes equation.} 
From now on, the letter $C$ denotes generic positive constants that may depend only on the choice of $\Omega$, the number $\delta$ from \eqref{SEP:DELTA}, the initial data and the system parameters, but not on $\eps$.
The exact value of $C$ may vary throughout this proof.

Testing \eqref{C:ModelH:NS} by $A_S^{-1}\tvv$ and invoking the identity 
\begin{align*}
    \frac{1}{2} \ddt\|\tvv\|_{\sigma}^2
     = \frac{1}{2} \ddt \norm{\Grad A_S^{-1}\tvv\big}^2 
	= ( \delt \tvv , A_S^{-1} \tvv ) ,
\end{align*}
we obtain
\begin{align}\label{C:ModelH:NS_1}
    \begin{aligned}
	&\frac{1}{2} \ddt\|\tvv\|_{\sigma}^2 
    +  \intO \big[(\vveps\cdot\Grad)\vveps- (\vv\cdot\Grad)\vv \big]\cdot A_S^{-1}\tvv\dx 
    +  \intO \Grad \tvv:\Grad A_S^{-1}\tvv\dx 
    \\&
	= \intO(\mueps\Grad \ceps - \mu\Grad c)\cdot A_S^{-1}\tvv\dx.
    \end{aligned}
\end{align}
Recalling that $\vv$, $\vveps$ and $\tvv$ are divergence-free, the second term on the left-hand side of \eqref{C:ModelH:NS_1} can be reformulated as
\begin{align}\label{NScalc1}
    \begin{aligned}
	& \intO \big[(\vveps\cdot\Grad)\vveps- (\vv\cdot\Grad)\vv \big]\cdot A_S^{-1}\tvv\dx 
    \\
	&=  \intO(\vveps\cdot\Grad)\tvv\cdot A_S^{-1}\tvv\dx
	   +  \intO(\tvv\cdot\Grad)\vv\cdot A_S^{-1}\tvv\dx
    \\
    &= -  \intO(\vveps\otimes \tvv): \Grad A_S^{-1}\tvv\dx
	   -  \intO(\tvv\otimes \vv): \Grad A_S^{-1}\tvv\dx.
    \end{aligned}
\end{align}
Using the Gagliardo--Nirenberg inequality, Young's inequality and the uniform bound \eqref{regg-est1}, the first term can be estimated as
\begin{align}\label{NScalc2a}
    \begin{aligned}
	&\abs{ \intO(\vveps\otimes \tvv): \Grad A_S^{-1}\tvv\dx }
	\leq \, \|\vveps\|_{\LL^6(\Omega)}\|\Grad A_S^{-1}\tvv\|_{\LL^3(\Omega)}\|\tvv\| 
    \\
	&\quad \leq C\|\vveps\|_{\Vsigma}\|\Grad A_S^{-1}\tvv\|^{\frac12}\norm{\tvv}^\frac32
    \leq \frac{1}{16} \|\tvv\|^2 + C\|\tvv\|_{\sigma}^2.
    \end{aligned}
\end{align}
Proceeding similarly, we deduce
\begin{align}\label{NScalc2b}
    \begin{aligned}
	&\abs{ \intO(\tvv\cdot\Grad)\vv\cdot A_S^{-1}\tvv\dx }
	\leq \, \|\tvv\|\|\vv\|_{\LL^{6}(\Omega)}\|\Grad A_S^{-1}\tvv\|_{\LL^3(\Omega)} 
    \\
	&\quad \leq C \|\tvv\|\norm{\vv}_{\Vsigma}\|\Grad A_S^{-1}\tvv\|^{\frac12}\|\Grad A_S^{-1}\tvv\|_{{\Vsigma}}^{\frac12} 
    \\
	&\quad \leq C\|\tvv\|^{\frac32} \|\tvv\|_{\sigma}^{\frac12} 
        \leq \frac{1}{16} \|\tvv\|^2 + C \|\tvv\|_{\sigma}^2.
 \end{aligned}
\end{align}
Combining \eqref{NScalc1}--\eqref{NScalc2b}, we conclude
\begin{align}\label{NScalc3}
	 \abs{\intO \big[(\vveps\cdot\Grad)\vveps- (\vv\cdot\Grad)\vv \big]\cdot A_S^{-1}\tvv\dx}
    \leq\frac{1}{8} \|\tvv\|^2 + C\|\tvv\|_{\sigma}^2.
\end{align}
Using integration by parts and recalling once more that $\tvv$ is divergence-free, we further obtain
\begin{align}\label{NScalc4}
	 \intO \Grad \tvv:\Grad A_S^{-1}\tvv\dx 
    = \|\tvv\|^2.
\end{align}
Via integration by parts, the right-hand side of \eqref{C:ModelH:NS_1} can be reformulated as
\begin{align}\label{NScalc5}
	\intO(\mueps\Grad \ceps - \mu\Grad c)\cdot A_S^{-1}\tvv\dx 
    = - \intO\Grad\mueps \tc\cdot A_S^{-1}\tvv\dx + \intO\tmu\Grad c\cdot A_S^{-1}\tvv\dx.
\end{align} 
Employing the Gagliardo--Nirenberg inequality, Young's inequality and the uniform bound \eqref{regg-est1}, the first term can be estimated as
\begin{align}\label{NScalc6}
    \begin{aligned}
	&\abs{\intO\Grad\mueps \tc\cdot A_S^{-1}\tvv\dx }
    \leq \|\Grad\mueps\|\|\tc\|\|A_S^{-1}\tvv\|_{\LL^\infty(\Omega)} 
    \\
	&\quad\leq C \|\Grad\mueps\| \norm{\tc}\|A_S^{-1}\tvv\|_{\HH^1(\Omega)}^{\frac12} 
        \|A_S^{-1}\tvv\|_{\HH^2(\Omega)}^{\frac12} 
    \leq C \norm{\tc}\|\tvv\|_{\sigma}^{\frac12}\|\tvv\|^{\frac12}
    \\
	&\quad\leq  \frac{\theta_0}{16} \norm{\tc}^2 
        + \frac{1}{8} \|\tvv\|^2
        + C \|\tvv\|_{\sigma}^2
    \\
    &\quad\leq \frac{\theta_0}{8} \norm{\tc - \overline{\tc}}^2 
        + \frac{1}{8} \|\tvv\|^2
        + C|\overline{\tc}|^2 
        + C \|\tvv\|_{\sigma}^2.
    \end{aligned}
\end{align}
We point out that this estimate is one of the main reasons for which the solution $(\vveps,\ceps)$ has to be strong since otherwise we would only have $\mueps\in L^2(0,T;H^1(\Omega))$.
By means of \eqref{C:ModelH:CH2}, the second term in \eqref{NScalc5} can be expanded as
\begin{align}\label{NScalc7}
    \begin{aligned}
	\intO\tmu\Grad c\cdot A_S^{-1}\tvv\dx
    &= \intO\big(\mathcal{L}_\eps c + \Delta c\big)\Grad c\cdot A_S^{-1}\tvv\dx
	+ \intO\mathcal{L}_\eps\tc\:\Grad c\cdot A_S^{-1}\tvv\dx
    \\
	&\quad +\intO\big(f^\prime(\ceps) - f^\prime(c)\big)\Grad c\cdot A_S^{-1}\tvv\dx.
    \end{aligned}
\end{align}
Using Young's inequality, Agmon's inequality and the uniform estimate \eqref{regg-est1}, the first summand on the right-hand side can be estimated as
\begin{align}\label{NScalc8}
    \begin{aligned}
	&\intO\big(\mathcal{L}_\eps c 
        + \Delta c\big)\Grad c\cdot A_S^{-1}\tvv\dx
    \\
    &\quad\leq \frac{1}{2}\|\mathcal{L}_\eps c + \Delta c\|^2 
        + \frac{1}{2}\|\Grad c\|^2\|A_S^{-1}\tvv\|^2_{\LL^\infty(\Omega)} 
    \\
	&\quad\leq \frac{1}{2}\|\mathcal{L}_\eps c + \Delta c\|^2 
        + C\|\Grad c\|^2\|A_S^{-1}\tvv\|_{\HH^1(\Omega)} \|A_S^{-1}\tvv\|_{\HH^2(\Omega)} 
    \\
	&\quad\leq \frac{1}{2}\|\mathcal{L}_\eps c + \Delta c\|^2 + C \|\tvv\|_{\sigma}\|\tvv\|
    \\
    &\quad\leq \frac{1}{2}\|\mathcal{L}_\eps c + \Delta c\|^2 + \frac{1}{8} \|\tvv\|^2 + C\|\tvv\|_{\sigma}^2,
    \end{aligned}
\end{align}
where we exploited \eqref{reggloc} for the $L^\infty(0,T;H^1(\Omega))$-regularity of $\mu$.
To estimate the second summand on the right-hand side of \eqref{NScalc7},  
we use the folowing Poincaré type inequality, which can be found in \cite[Theorem 1]{Ponce}: there exists $C>0$, such that for all $ f\in H^1(\Omega)$,
\begin{align}
\mathcal{E}_\eps(f)\leq C\norm{f}_{H^1(\Omega)}.
    \label{poincare}
\end{align}
We further recall \eqref{regularity_c}, \eqref{regularity_c2} and \eqref{regg-est1}, the definition of $\mathcal{L}_\eps$, and the continuous embeddings
$H^3(\Omega)\emb W^{2,4}(\Omega) \emb W^{1,\infty}(\Omega) \emb W^{1,4}(\Omega)$ and $\HH^1(\Omega)\emb\LL^4(\Omega)$. With the help of these results, we derive the estimate
\begin{align}
    \label{NScalc9}
    \begin{aligned}
	&\intO\mathcal{L}_\eps\tc\:\Grad c\cdot A_S^{-1}\tvv\dx
    \\
    &\quad \leq 2\sqrt{\mathcal{E}_\eps(\tc)}\, \sqrt{\mathcal{E}_\eps(\Grad c\cdot A_S^{-1}\tvv)}
    \leq 
    C\sqrt{\mathcal{E}_\eps(\tc)}\, \norm{\Grad c\cdot A_S^{-1}\tvv}_{H^1(\Omega)} 
    \\
    &\quad\leq C\sqrt{\mathcal{E}_\eps(\tc)}\,
    \Big(\norm{\Grad c}_{\LL^4(\Omega)}\norm{ A_S^{-1}\tvv}_{\LL^4(\Omega)}
        +\norm{D^2 c}_{\LL^4(\Omega)}\norm{A_S^{-1}\tvv}_{\LL^4(\Omega)}
    \\
    &\qquad\qquad\qquad\qquad
        +\norm{\Grad c}_{\LL^\infty(\Omega)}\norm{\Grad A_S^{-1}\tvv}\Big)
    \\
    &\quad\leq C\sqrt{\mathcal{E}_\eps(\tc)}\,
    \norm{c}_{H^3(\Omega)}
    \norm{A_S^{-1}\tvv}_{\HH^1(\Omega)}
    \\
	&\quad\leq {\frac{1}{2}}\mathcal{E}_\eps(\tc) 
        + C\norm{\tvv}_\sigma^2
    = {\frac{1}{2}}\mathcal{E}_\eps\big(\tc- \ov{\tc}\big) 
        + C\norm{\tvv}_\sigma^2,
    \end{aligned}
\end{align}
recalling $\mathcal{E}_\eps(\tc) = \mathcal{E}_\eps\big(\tc-\ov{\tc}\big)$.
Here, the first inequality follows by exploiting the properties of the interaction kernel $J_\eps$ (cf.~\cite[p.\,128]{DST3}).
Invoking \eqref{contt2}, Agmon's inequality and Young's inequality, the third summand on the right-hand side of \eqref{NScalc7} can be bounded via the estimate
\begin{align}
    \begin{aligned}
    &\intO\big(f^\prime(\ceps)-f^\prime(c)\big)\Grad c\cdot A_S^{-1}\tvv\dx 
    \\
    &\quad\leq C\|f^\prime(\ceps)-f^\prime(c)\|_{L^1(\Omega)}\|\Grad c\|_{\LL^\infty(\Omega)}\|\tvv\|_{\sigma}^{\frac12}
        \|\tvv\|^{\frac12}
    \\
    &\quad \leq C K_\delta \norm{\ceps-c}\, 
        \|\Grad A_S^{-1}\tvv\|^{\frac12}\|A_S^{-1}\tvv\|_{\HH^2(\Omega)}^{\frac12}
    \\
    &\quad \leq \frac{1}{8} \norm{\tvv}^2+\frac{\theta_0}{16}\norm{\tc}^2+CK_\delta^4\norm{\tvv}_{\sigma}^2
    \\
    &\quad =\frac{1}{8} \norm{\tvv}^2+\frac{\theta_0}{2}\norm{\tc - \ov{\tc}}^2 +C|\overline{\tc}|^2 
        + C\norm{\tvv}_{\sigma}^2.
    \end{aligned}
    \label{NScalc10}
\end{align}
In view of \eqref{NScalc7},\eqref{NScalc8}, \eqref{NScalc9} and \eqref{NScalc10}, we thus have
\begin{align}
    \label{NScalc11}
    \abs{\intO\tmu\Grad c\cdot A_S^{-1}\tvv\dx} 
    \le {\frac{1}{2}}\mathcal{E}_\eps\big(\tc-\ov{\tc}\big)
        + \frac{1}{4} \norm{\tvv}^2
        + \frac{\theta_0}{2}\norm{\tc - \ov{\tc}}^2 
        + C\norm{\tvv}_{\sigma}^2
        + C|\overline{\tc}|^2 .
\end{align}
Eventually, combining \eqref{C:ModelH:NS_1} with \eqref{NScalc3}, \eqref{NScalc4} and \eqref{NScalc11}, we conclude
\begin{align}
    \label{NScalc12}
    \begin{aligned}
	& \frac{1}{2} \ddt \|\tvv\|_{\sigma}^2 + \frac{3}{4}\|\tvv\|^2
    \\
    &\quad \leq C\big(\|\tvv\|_{\sigma}^2 + \|\tc-\overline{\tc}\|_{\ast}^2 \big) 
        + \frac1{2} \mathcal{E}_\eps\big(\tc- \ov{\tc}\big) 
        + \frac{\theta_0}{2}\norm{\tc-\overline{\tc}}^2 
        + \frac{1}{2}\|\mathcal{L}_\eps c 
        + \Delta c\|^2 
        + C|\overline{\tc}|^2.
    \end{aligned}
\end{align}

\paragraph{Step~3: Estimates for the convective Cahn--Hilliard system.} 

Testing \eqref{C:ModelH:CH1} with $\mathcal{N}(\tc - \overline{\tc})$ and using the identity
\begin{equation*}
    \frac{1}{2}\ddt\|\tc-\overline{\tc}\|_{\ast}^2 
    = \frac{1}{2}\ddt\|\Grad\mathcal{N}\big(\tc-\overline{\tc}\big)\|^2
    = \big( \delt \tc , \mathcal{N}(\tc-\overline{\tc}) \big),
\end{equation*}
we derive the equation
\begin{align}\label{CHcalc1}
	\frac{1}{2}\ddt\|\tc-\overline{\tc}\|_{\ast}^2 
    = -\intO\big(\vveps\cdot\Grad \ceps - \vv\cdot\Grad c\big)
            \mathcal{N}(\tc-\overline{\tc})\dx 
        - \intO\tmu(\tc-\overline{\tc})\dx,
\end{align}
Expressing $\tmu$ via \eqref{C:ModelH:CH2}, the second term on the right-hand side can be reformulated as
\begin{align}\label{CHcalc2}
    \begin{aligned}
	- \intO \tmu(\tc-\overline{\tc})\dx 
    &= - \intO\big(\mathcal{L}_\eps \ceps 
        + \Delta c\big)(\tc-\overline{\tc})\dx 
    \\
    &\qquad - \intO\big(f^\prime(\ceps) - f^\prime(c)\big)(\tc-\overline{\tc})\dx.
    \end{aligned}
\end{align}
Recalling the definition of $f$ and the condition $F''\ge \theta$ 
(see~\ref{ASS:S1}), we use Young's inequality along with \eqref{contt2} to obtain
\begin{align}
    \label{CHcalc3}
    \begin{aligned}
    &- \intO\big(f^\prime(\ceps) - f^\prime(c)\big)(\tc-\overline{\tc})\dx 
    \\
    &= - \intO\big(F^\prime(\ceps) - F^\prime(c)\big)\tc\dx 
        + \theta_0 \norm{\tc}^2
        -  \intO\big(f^\prime(\ceps) - f^\prime(c)\big)\overline{\tc}\dx 
    \\[1ex]
    &\leq - \theta\norm{\tc}^2 
        +  \theta_0\norm{\tc}^2 
        + |\overline{\tc}|\;\|f^\prime(c_\eps) - f^\prime(c)\|_{L^1(\Omega)} 
    \\
    &\leq - \theta\norm{\tc}^2 
        + \frac{9}{8}\theta_0 \norm{\tc- \overline{\tc}}^2 
        + C|\overline{\tc}|^2.
    \end{aligned}
\end{align}
We next use the identity
\begin{align}\label{CHcalc4}
	\intO\mathcal{L}_\eps \tc\, (\tc-\overline{\tc})\dx 
    = 2\mathcal{E}_\eps(\tc -\overline{\tc} ),
\end{align}
which follows by a straightforward computation exploiting the symmetry of the interaction kernel $J_\eps$.
Using this result, we deduce
\begin{align}\label{CHcalc6}
    \begin{aligned}
    - \intO\big(\mathcal{L}_\eps \ceps 
        + \Delta c\big)(\tc-\overline{\tc})\dx 
    &= - \intO\big(\mathcal{L}_\eps c + \Delta c\big)(\tc-\overline{\tc})\dx 
        - \intO\mathcal{L}_\eps \tc(\tc-\overline{\tc})\dx 
    \\
    &\leq \frac{1}{2}\norm{\mathcal{L}_\eps c + \Delta c}^2 
        + \frac{\theta_0}{8}\norm{\tc 
        - \overline{\tc}}^2 
        - 2\mathcal{E}_\eps(\tc-\overline{\tc}).
    \end{aligned}
\end{align}
Combining \eqref{CHcalc2}, \eqref{CHcalc2} and \eqref{CHcalc6}, we have
\begin{align}
    \label{CHcalc7}
    - \intO \tmu(\tc-\overline{\tc})\dx 
    \leq \frac{5}{4}\theta_0 \norm{\tc- \overline{\tc}}^2 
        + \frac{1}{2}\norm{\mathcal{L}_\eps c +\Delta c}^2
        + C|\overline{\tc}|^2
        - 2\mathcal{E}_\eps(\tc-\overline{\tc})
        - \theta\norm{\tc}^2 .
\end{align}
Recalling that $\vveps$ and $\vv$ are divergence-free and using integration by parts, the second summand on the right-hand side of \eqref{CHcalc1} can be expressed as
\begin{align}
    \label{CHcalc8}
    \begin{aligned}
	&\intO\big(\vveps\cdot\Grad \ceps - \vv\cdot\Grad c\big)\mathcal{N}(\tc-\overline{\tc})\dx 
    \\
    &\quad = - \intO\vveps\tc\cdot\Grad\mathcal{N}(\tc-\overline{\tc})\dx 
        - \intO c\:\tvv \cdot\Grad\mathcal{N}(\tc-\overline{\tc})\dx.
    \end{aligned}
\end{align}
Invoking H\"older's inequality, the Gagliardo--Nirengberg inequality, the embedding $\Vsigma\hookrightarrow \LL^4(\Omega)$ and the uniform bound 
\eqref{regg-est1}, the first term on the right-hand side can be estimated as
\begin{align}
    \label{CHcalc9}
    \begin{aligned}
	&\abs{\intO\vveps\tc\cdot\Grad\mathcal{N}(\tc-\overline{\tc})\dx} 
    \leq \|\vveps\|_{\LL^6(\Omega)}\|\tc\|\|\Grad\mathcal{N}(\tc-\overline{\tc})\|_{\LL^3(\Omega)} 
    \\
	&\quad\leq C\|\vveps\|_{\Vsigma} \|\tc-\overline{\tc}\|^{\frac32} 
        \|\tc-\overline{\tc}\|_{\ast}^{\frac12} 
    + C|\overline{\tc}|\,\|\vveps\|_{\Vsigma}\|
    \|\tc-\overline{\tc}\|^{\frac12}\|\tc-\overline{\tc}\|_{\ast}^{\frac12}
    \\
	&\quad\leq \frac{\theta_0}{4}\|\tc-\overline{\tc}\|^2 
    + C\|\tc-\overline{\tc}\|_{\ast}^2 + C|\overline{\tc}|^2. 
    \end{aligned}
\end{align}
Moreover, employing \eqref{regg-est1} and H\"older's inequality, we show that the second summand on the right-hand side of \eqref{CHcalc8} fulfills the estimate
\begin{align}
    \label{CHcalc10}
	\abs{\intO c\:\tvv \cdot\Grad\mathcal{N}(\tc-\overline{\tc})\dx} 
    \leq \|\tvv\|\|c\|_{L^\infty(\Omega)}\|\Grad\mathcal{N}(\tc-\overline{\tc})\| 
	\leq \frac{1}{4}\|\tvv\|^2 + C\|\tc-\overline{\tc}\|_{\ast}^2.
\end{align}
Using \eqref{CHcalc9} and \eqref{CHcalc10} to estimate the right-hand side of \eqref{CHcalc8}, we infer
\begin{align}
    \label{CHcalc10*}
    \begin{aligned}
	- \intO\big(\vveps\cdot\Grad \ceps - \vv\cdot\Grad c\big)\mathcal{N}(\tc-\overline{\tc})\dx
    \le 
    \frac{\theta_0}{4}\|\tc-\overline{\tc}\|^2 
        + \frac{1}{4}\|\tvv\|^2 
        + C|\overline{\tc}|^2
        + C\|\tc-\overline{\tc}\|_{\ast}^2.
    \end{aligned}
\end{align}
Eventually, using \eqref{CHcalc7} and \eqref{CHcalc10*} to bound the right-hand side of \eqref{CHcalc1}, we conclude
\begin{align}
    \label{CHcalc11}
    \begin{aligned}
	&\ddt\frac{1}{2}\|\tc-\overline{\tc}\|_{\ast}^2 
        + \theta\|\tc\|^2
        + 2\mathcal{E}_\eps(\tc -\overline{\tc}) 
    \\
    &\leq  
	    \frac{1}{2}\|\mathcal{L}_\eps c 
        + \Delta c\|_{L^2(\Omega)}^2 
        + \frac{1}{4}\|\tvv\|^2 
        + \frac{3}{2}\theta_0 \norm{\tc -\overline{\tc}}^2 
        + C|\overline{\tc}|^2
        + C\|\tc-\overline{\tc}\|_{\ast}^2.
    \end{aligned}
\end{align}

\paragraph{Step 4: Completion of the proof.} 
Adding \eqref{NScalc11} and \eqref{CHcalc11} we obtain
\begin{align}\label{CHcalc12}
    \begin{aligned}
	&\ddt\Big(\frac{1}{2}\|\tvv\|_{\sigma}^2 
        + \frac{1}{2}\|\tc-\overline{\tc}\|_{\ast}^2\Big) 
        + \frac{1}{2}\|\tvv\|^2 
        + \theta\|\tc\|^2
        + \frac{3}{2}\mathcal{E}_\eps\big(\tc- \ov{\tc}\big) 
    \\
	&\quad \leq C\Big(\|\tvv\|_{\sigma}^2 
        + \|\tc-\overline{\tc}\|_{\ast}^2\Big) 
	    + \|\mathcal{L}_\eps c 
        + \Delta c\|^2 
        + 2\theta_0\|\tc-\overline{\tc}\|^2 
        + C|\overline{\tc}|^2
    \end{aligned}
\end{align}
in $[0,T]$.
Recalling the definition of $\eps_s$ in \eqref{poinn}, 
applying Lemma~\ref{LEM:EN}\ref{EN:IEQ:2} with $\gamma=\frac{1}{2\theta_0}$ yields
\begin{align*}
    2\theta_0\|\tc-\overline{\tc}\|^2 \leq \mathcal{E}_\eps(\tc-\overline{\tc}) + C\|\tc-\overline{\tc}\|_{\ast}^2
\end{align*}
as we consider $\eps\in (0,\eps_s]$. Plugging this estimate into \eqref{CHcalc12} and recalling 
\begin{equation*}
    \ov{\tc}(t) = \ov{\tc_0} \;\;\text{for all $t\in[0,T]$}
    \quad\text{and}\quad
    \mathcal{E}_\eps(\tc) = \mathcal{E}_\eps\big(\tc-\ov{\tc}\big),
\end{equation*}
we infer
\begin{align}\label{CHcalc13}
    \begin{aligned}
	&\ddt\Big(\frac{1}{2}\|\tvv\|_{\sigma}^2 
        + \frac{1}{2}\|\tc-\overline{\tc}\|_{\ast}^2\Big) 
        + \frac{1}{2}\|\tvv\|^2 
        + \theta\|\tc\|^2
        + \frac{1}{2}\mathcal{E}_\eps(\tc)
    \\
	&\quad \leq C\Big(\|\tvv\|_{\sigma}^2 
        + \|\tc-\overline{\tc}\|_{\ast}^2\Big) 
	    +\|\mathcal{L}_\eps c 
        + \Delta c\|^2 
        + C|\ov{\tc_0}|^2
    \end{aligned}
\end{align}
in $[0,T]$.
Thus, Gronwall's lemma implies
\begin{align}
    \begin{aligned}
	&\frac{1}{2}\sup_{t\in[0,T]}\|\tvv(t)\|_{\sigma}^2 
        + \frac{1}{2}\sup_{t\in[0,T]}\|\tc(t) 
        - \overline{\tc}(t)\|_{\ast}^2
    \\
    & + \theta\int_0^T\|\tc\|^2\dt  
        + \frac{1}{2} \int_0^T\norm{\tvv}^2\dt 
        + \frac 1{2}\int_{0}^T\mathcal{E}_\eps(\tc)\dt 
    \\
	&\leq \left(\frac{1}{2}\|\tvv_0\|_{\sigma}^2 
        + \frac{1}{2}\|\tc_0\|_{\ast}^2 
        + C\int_{0}^T|\overline{\tc_0}|^2\dt 
        + \int_{0}^T\|\mathcal{L}_\eps c 
        + \Delta c\|^2\dt\right)e^{CT}.
    \end{aligned}
\end{align}
Now, since $c\in L^\infty(0,T;H^3(\Omega))$ (see~\eqref{regularity_c} and \eqref{regularity_c2}), Proposition \ref{PROP:NLTL} yields
\begin{equation*}
    \int_{0}^T\|\mathcal{L}_\eps c + \Delta c\|^2\leq C\eps^{2\alpha}\norm{c}_{L^2(0,T;H^3(\Omega)}^2
    \leq C\eps^{2\alpha},
\end{equation*}
where $\alpha=\frac 12$ in case $\Omega$ is a bounded domain and $\alpha=1$ if $\Omega=\TTn$. Together with assumption \eqref{initiald:2}, we thus have
\begin{align}
    \label{CR:1}
	&\nonumber\frac{1}{2}\sup_{t\in[0,T]}\|\tvv(t)\|_{\sigma}^2 
        + \frac{1}{2}\sup_{t\in[0,T]}\|\tc(t) - \overline{\tc}(t)\|_{\ast}^2
    \\
    & + \theta\|\tc\|_{L^2(0,T;L^2(\Omega))}^2  
        + \frac{1}{2} \norm{\tvv}_{L^2(0,T;L^2(\Omega))}^2
        + \frac 1{2}\int_{0}^T\mathcal{E}_\eps(\tc)\dt 
	\leq C \eps^{2\alpha}.
\end{align}
As the norms $\norm{\cdot}_\sigma$ and $\norm{\cdot}_{\Vsigma'}$ on $\Vsigma'$ are equivalent, this also yields
\begin{equation}
    \label{CR:2}
    \norm{\tvv}_{L^\infty(0,T;\Vsigma')} 
    \leq C \sup_{t\in[0,T]}\|\tvv(t)\|_{\sigma}^2
    \leq C \eps^{\alpha}.
\end{equation}
Moreover, since the norms $\norm{\cdot}_*$ and $\norm{\cdot}_{H^1(\Omega)'}$ on $H_{(0)}^{-1}(\Omega)$ are equivalent, we further have
\begin{equation}
    \label{CR:3}
    \begin{aligned}
    \norm{\tc}_{L^\infty(0,T;H^1(\Omega)')} 
    &\leq \norm{\tc - \ov{\tc}}_{L^\infty(0,T;H^1(\Omega)')} + \norm{\ov{\tc}}_{L^\infty(0,T)} 
    \\
    &\leq \sup_{t\in[0,T]}\|\tc(t) - \overline{\tc}(t)\|_{\ast}^2 +  \abs{\ov{\tc_0}}
    \leq C \eps^{\alpha}.
    \end{aligned}    
\end{equation}
Combining \eqref{CR:1}--\eqref{CR:3}, we have thus verified estimate \eqref{finale}. Hence, the proof is complete.
$\phantom{x}$\hfill$\Box$



\section*{Acknowledgement}

The authors want to warmly thank Helmut Abels for helpful and inspiring discussions. In particular, part of this work was done while AP was visiting Helmut Abels and Harald Garcke at
the Faculty for Mathematics of the University of Regensburg, whose hospitality is kindly acknowledged. 

CH and PK were partially supported by the RTG 2339 ``Interfaces, Complex Structures, and Singular Limits'' of the Deutsche Forschungsgemeinschaft (DFG, German Research Foundation). Moreover, PK was partially supported by the project 524694286 of the Deutsche Forschungsgemeinschaft (DFG, German Research Foundation).  

AP is  a member of Gruppo Nazionale per l’Analisi Matematica, la Probabilità e le loro Applicazioni (GNAMPA) of
Istituto Nazionale per l’Alta Matematica (INdAM). This research was funded in whole or in part by the Austrian Science Fund (FWF) \href{https://doi.org/10.55776/ESP552}{10.55776/ESP552}. 

For open access purposes, the authors have applied a CC BY public copyright license to any author accepted manuscript version arising from this submission.



\footnotesize

\bibliographystyle{abbrv}
\bibliography{HKP}

\end{document}